\documentclass[11pt]{amsart}
\usepackage{graphicx}
\usepackage{amsmath}
\usepackage{amssymb}
\usepackage{mathrsfs}
\usepackage[tight]{subfigure}
\usepackage{epsfig,epstopdf,psfrag}
\usepackage{hyperref,bm,color,fancyhdr}
\usepackage{algorithm,algorithmic}
\usepackage{array,mathtools}
\usepackage{booktabs}
\usepackage{natbib}
\newtheorem{theorem}{Theorem}
\newtheorem{lemma}[theorem]{Lemma}
\newtheorem{prop}[theorem]{Proposition}

\newcommand{\PV}{\mathop{\rm P.V.}\nolimits}
\newcommand{\re}{\mathrm{e}}

\providecommand{\abs}[1]{\lvert#1\rvert}
\providecommand{\norm}[1]{\lVert#1\rVert}

\numberwithin{equation}{section}
\numberwithin{theorem}{section}
\numberwithin{figure}{subsection}
\numberwithin{table}{subsection}

\title[Blowup of a 1D Model for 3D Euler]{On the Finite-Time Blowup of a 1D Model for the 3D Incompressible Euler Equations}
\author{Thomas Y. Hou$^{\dag}$ \and Guo Luo$^{\dag}$}
\thanks{$\dag$: Applied and Computational Mathematics, California Institute of Technology.}
\begin{document}
\maketitle

\begin{center}
\today
\end{center}

\section*{Abstract}
We study a 1D model for the 3D incompressible Euler equations in axisymmetric geometries, which can be viewed as a local approximation to the Euler equations near the solid boundary of a cylindrical domain. We prove the local well-posedness of the model
in spaces of zero-mean functions, and study the potential formation of a finite-time singularity under certain convexity conditions for the velocity field. It is hoped that the results obtained on the 1D model will be useful in the analysis of the full
3D problem, whose loss of regularity in finite time has been observed in a recent numerical study \citep{lh2013}.

\section{Background}\label{sec_intro}
The purpose of this note is to summarize some of the results we obtained on a 1D model for the 3D incompressible Euler equations. In a recently completed computation \citep{lh2013}, we have numerically studied the 3D Euler equations in axisymmetric
geometries and identified a class of potentially singular solutions. The equations being solved take the form
\begin{subequations}\label{eqn_eat}
\begin{align}
  u_{1,t} + u^{r} u_{1,r} + u^{z} u_{1,z} & = 2 u_{1} \psi_{1,z}, \label{eqn_eat_u} \\
  \omega_{1,t} + u^{r} \omega_{1,r} + u^{z} \omega_{1,z} & = (u_{1}^{2})_{z}, \label{eqn_eat_w} \\
  -\bigl[ \partial_{r}^{2} + (3/r) \partial_{r} + \partial_{z}^{2} \bigr] \psi_{1} & = \omega_{1}, \label{eqn_eat_psi}
\end{align}
\end{subequations}
where
\begin{displaymath}
  u_{1} = u^{\theta}/r,\qquad \omega_{1} = \omega^{\theta}/r,\qquad \psi_{1} = \psi^{\theta}/r,
\end{displaymath}
are transformed angular velocity, vorticity, and stream functions and
\begin{displaymath}
  u^{r} = -r \psi_{1,z},\qquad u^{z} = 2\psi_{1} + r \psi_{1,r},
\end{displaymath}
are radial and axial velocity components. The solutions of \eqref{eqn_eat} were computed in the cylinder
\begin{displaymath}
  D(1,L) = \Bigl\{ (r,z)\colon 0 \leq r \leq 1,\ 0 \leq z \leq L \Bigr\},
\end{displaymath}
with carefully chosen initial data and no-flow (in $r$) and periodic (in $z$) boundary conditions. It was observed that the vorticity function $\abs{\omega}$ develops a point singularity in finite time at the corner $\tilde{q}_{0} = (1,0)^{T}$, which
corresponds to a ``singularity ring'' on the solid boundary of the cylinder. The numerical data has been carefully checked against all major blowup (non-blowup) criteria including Beale-Kato-Majda \citep{bkm1984}, Constantin-Fefferman-Majda
\citep{cfm1996}, and Deng-Hou-Yu \citep{dhy2005}, to confirm the validity of the singularity. A local analysis near the point of the singularity also suggests the existence of a self-similar blowup. The interested readers are referred to \citet{lh2013}
for more details.

\section{The 1D Model and Its Well-Posedness}\label{sec_model}
For the particular solution studied in \citet{lh2013}, it is observed that, near the point of the singularity $\tilde{q}_{0} = (1,0)^{T}$, the axial velocity $u^{z}$ is negative when $z > 0$ and positive when $z < 0$. This creates a compression mechanism
along the wall which seems to be responsible for the generation of the finite-time singularity. Motivated by these observations, we consider in this note the following 1D model
\begin{subequations}\label{eqn_eat_1d}
\begin{align}
  u_{t} + v u_{z} & = 0,\qquad z \in (0,L), \label{eqn_eat_1d_u} \\
  \omega_{t} + v \omega_{z} & = u_{z}, \label{eqn_eat_1d_w}
\end{align}
with the nonlocal, zero-mean velocity $v$ determined by
\begin{equation}
  v_{z}(z) = H\omega(z) := \frac{1}{L} \PV\int_{0}^{L} \omega(y) \cot \bigl[ \mu (z-y) \bigr]\,dy,\qquad \mu = \pi/L.
  \label{eqn_eat_1d_vz}
\end{equation}
\end{subequations}
The problem is complemented by periodic boundary conditions and zero-mean initial data.

This 1D model can be viewed as the ``restriction'' of the 3D axisymmetric Euler equations \eqref{eqn_eat} to the wall $r = 1$, with the identification
\begin{displaymath}
  u(z) \sim u_{1}^{2}(1,z),\qquad \omega(z) \sim \omega_{1}(1,z),\qquad v(z) \sim \psi_{1,r}(1,z).
\end{displaymath}
Indeed, the no-flow boundary condition ($\psi_{1}(1,z) = 0$) implies that
\begin{displaymath}
  u^{r} = -r \psi_{1,z} = 0\qquad \text{on}\qquad r = 1,
\end{displaymath}
hence the evolution equations \eqref{eqn_eat_u}--\eqref{eqn_eat_w} reduce to \eqref{eqn_eat_1d_u}--\eqref{eqn_eat_1d_w} on the wall. To define the velocity $v$, we observe that
\begin{displaymath}
  \psi_{1,r}(r,z) \ll \psi_{1,rr}(r,z),\qquad \omega_{1}(r,z) \approx \omega_{1}(1,z),
\end{displaymath}
near the point of the singularity \citep{lh2013}. Hence the Poisson equation \eqref{eqn_eat_psi} can be locally approximated by
\begin{displaymath}
  -\bigl[ \partial_{r}^{2} + \partial_{z}^{2} \bigr] \psi_{1} = \omega_{1}(1,z),
\end{displaymath}
the solution of which on the stretched domain $r \in (-\infty,1)$ satisfies
\begin{displaymath}
  \psi_{1,rz}(1,z) = H\omega_{1}(1,z).
\end{displaymath}
This is precisely equation \eqref{eqn_eat_1d_vz} which provides the key relation needed to close \eqref{eqn_eat_1d_u}--\eqref{eqn_eat_1d_w}.

Problems similar to \eqref{eqn_eat_1d_w} have been studied in the past as models for the 3D Euler equations. In \citet{clm1985}, the equation
\begin{subequations}\label{eqn_m}
\begin{equation}
  \omega_{t} - v_{x} \omega = 0,\qquad v_{x} = H\omega,
  \label{eqn_m_clm}
\end{equation}
was introduced as a model for the dynamics of vorticity in incompressible inviscid flows. The finite-time blowup of \eqref{eqn_m_clm} was established for a large class of initial data as a consequence of the explicit solution formula
\begin{displaymath}
  \omega(x,t) = \frac{4\omega_{0}(x)}{[2-tH\omega_{0}(x)]^{2} + t^{2} \omega_{0}^{2}(x)}.
\end{displaymath}
In \citet{degreg1990,degreg1996}, the model \eqref{eqn_m_clm} was modified to include a convection term:
\begin{equation}
  \omega_{t} + v\omega_{x} - v_{x} \omega = 0,\qquad v_{x} = H\omega,
  \label{eqn_m_degreg}
\end{equation}
and the resulting problem was conjectured to admit globally regular solutions. In \citet{ccf2005}, the equation
\begin{equation}
  \theta_{t} + \theta_{x} H\theta = 0,
  \label{eqn_m_ccf}
\end{equation}
was proposed as the simplest model for transport equations with a nonlocal velocity. The finite-time blowup of \eqref{eqn_m_ccf} was rigorously proved for a large class of initial data as a consequence of the estimate
\begin{displaymath}
  -\int_{0}^{\infty} \frac{f_{x}(x) Hf(x)}{x^{1+\delta}}\,dx \geq C_{\delta} \int_{0}^{\infty} \frac{f^{2}(x)}{x^{2+\delta}}\,dx,\qquad \forall \delta \in (0,1),
\end{displaymath}
which holds true for any even function $f$ decaying sufficiently fast at $\infty$ and vanishing at 0 \citep[see also][]{ccf2006}. In \citet{osw2008}, a generalization of the models \eqref{eqn_m_clm}--\eqref{eqn_m_ccf}:
\begin{equation}
  \omega_{t} + av\omega_{x} - v_{x} \omega = 0,\qquad v_{x} = H\omega,
  \label{eqn_m_osw}
\end{equation}
was studied. The model reduces to \eqref{eqn_m_clm} if $a = 0$, to \eqref{eqn_m_degreg} if $a = 1$, and to \eqref{eqn_m_ccf} if $a = -1$ and $\omega = -\theta_{x}$. The global regularity of \eqref{eqn_m_osw} was numerically demonstrated for the case $a =
1$ and was rigorously proved in the limit of $a \to \infty$, in which case \eqref{eqn_m_osw} reduces to
\begin{displaymath}
  \omega_{t} + v\omega_{x} = 0,\qquad v_{x} = H\omega.
\end{displaymath}
Other similar models were also proposed and analyzed in the literature. The interested readers are referred to \citet{cccf2005,cc2008,cc2010,ccg2010} for further readings.
\end{subequations}

Compared with the existing models, the 1D model \eqref{eqn_eat_1d} is distinct in that it consists of a \emph{system} of two equations while all other models considered so far are \emph{scalar} equations. In addition, the 1D model \eqref{eqn_eat_1d}
provides a natural approximation to the 3D axisymmetric Euler equations along the wall $r = 1$, while no such explicit connection exists in other models. The purpose of this note is to study the basic properties of \eqref{eqn_eat_1d} including its
(local) well-posedness and potential finite-time blowup. It is hoped that the results will be useful in the analysis of the full problem \eqref{eqn_eat}.

To study the well-posedness of the 1D model \eqref{eqn_eat_1d}, define
\begin{displaymath}
  V^{k}(S) = \Bigl\{ f\colon f \in H^{k}(S),\ \bar{f} = 0 \Bigr\},\qquad k \geq 0,
\end{displaymath}
where $S$ denotes the circle on the plane with circumference $L$, $H^{k}(S)$ the usual (real) Sobolev space on $S$, and
\begin{displaymath}
  \bar{f} := \frac{1}{L} \int_{0}^{L} f(z)\,dz
\end{displaymath}
the mean of $f$ on $S$. In view of the zero-mean property of functions in $V^{k}(S)$ and Poincar\'{e}'s inequality (see \eqref{eqn_cal_fact_p}), a suitable norm on $V^{k}(S)$ can be chosen as
\begin{displaymath}
  \norm{f}_{V^{k}} = \biggl[ \int_{0}^{L} \abs{\partial_{z}^{k} f(z)}^{2}\,dz \biggr]^{1/2},\qquad f \in V^{k}(S),
\end{displaymath}
with associated inner product
\begin{displaymath}
  (f,g)_{V^{k}} = \int_{0}^{L} \partial_{z}^{k} f(z) \cdot \partial_{z}^{k} g(z)\,dz,\qquad f,\, g \in V^{k}(S).
\end{displaymath}

The (local) well-posedness of the 1D model \eqref{eqn_eat_1d} is contained in the following three theorems.

\begin{theorem}[Local existence and uniqueness]
\begin{subequations}\label{eqn_1d_sol_sp}
Let $m \geq 1$ be any positive integer. For any initial data
\begin{equation}
  u_{0} \in V^{m+1}(S),\qquad \omega_{0} \in V^{m}(S),
  \label{eqn_1d_sol_sp_0}
\end{equation}
there exists $T > 0$ depending only on $\norm{u_{0}}_{V^{m+1}}$ and $\norm{\omega_{0}}_{V^{m}}$ such that the 1D model \eqref{eqn_eat_1d} has a unique solution
\begin{align}
  u & \in C([0,T]; V^{m+1}(S)) \cap C^{1}([0,T]; V^{m}(S)), \nonumber \\
  \omega & \in C([0,T]; V^{m}(S)) \cap C^{1}([0,T]; V^{m-1}(S)). \label{eqn_1d_sol_sp_t}
\end{align}
\end{subequations}
\label{thm_1d_loc_ex}
\end{theorem}

We say the solution $(u,\omega)$ belongs to class $CV^{m}$ on $[0,T]$ if it satisfies \eqref{eqn_1d_sol_sp_t}.

\begin{theorem}[Regularity]
Let $m \geq 1$ be any positive integer and let
\begin{displaymath}
  u \in C([0,T]; V^{2}(S)),\qquad \omega \in C([0,T]; V^{1}(S)),
\end{displaymath}
be a solution of \eqref{eqn_eat_1d} with initial data $u_{0} \in V^{m+1}(S),\ \omega_{0} \in V^{m}(S)$. Then
\begin{displaymath}
  u \in C([0,T]; V^{m+1}(S)),\qquad \omega \in C([0,T]; V^{m}(S)).
\end{displaymath}
In particular, $u(\cdot,t),\, \omega(\cdot,t) \in C^{\infty}(S)$ for each $t \in [0,T]$ if $u_{0},\, \omega_{0} \in C^{\infty}(S)$.
\label{thm_1d_reg}
\end{theorem}

In essence, the regularity theorem says that the existence interval $[0,T]$ of the solution depends only on the low-norm $\norm{u_{0}}_{V^{2}},\ \norm{\omega_{0}}_{V^{1}}$ of the initial data.

\begin{theorem}[Continuous dependence]
Let $m \geq 3$ be a positive integer and let
\begin{displaymath}
  u \in C([0,T]; V^{m+1}(S)),\qquad \omega \in C([0,T]; V^{m}(S)),
\end{displaymath}
be a solution of \eqref{eqn_eat_1d} with initial data $u_{0} \in V^{m+1}(S),\ \omega_{0} \in V^{m}(S)$. Let
\begin{displaymath}
  u_{0,j} \in V^{m+1}(S),\quad \omega_{0,j} \in V^{m}(S),\qquad j = 1,2,\dotsc,
\end{displaymath}
be a sequence of functions such that $u_{0,j} \to u_{0}$ in $V^{m+1}$ and $\omega_{0,j} \to \omega_{0}$ in $V^{m}$. Then there exists $T' \in (0,T]$ and solutions
\begin{displaymath}
  u_{j} \in C([0,T']; V^{m+1}(S)),\qquad \omega_{j} \in C([0,T']; V^{m}(S)),
\end{displaymath}
of \eqref{eqn_eat_1d} with initial data $(u_{0,j},\omega_{0,j})$ for sufficiently large $j$, such that
\begin{displaymath}
  u_{j} \to u\quad \text{in}\quad C([0,T']; V^{m+1}(S)),\qquad \omega_{j} \to \omega\quad \text{in}\quad C([0,T']; V^{m}(S)).
\end{displaymath}
\label{thm_1d_c_dep}
\end{theorem}

The local existence theorem (Theorem \ref{thm_1d_loc_ex}) is a direct consequence of an abstract existence theorem of \citet{kl1984} and various calculus inequalities. To prove Theorem \ref{thm_1d_reg}, we need the following energy estimate (see
Proposition \ref{prop_reg_hm})
\begin{displaymath}
  \max_{t \in [0,T]} \Bigl\{ \norm{u(\cdot,t)}_{V^{m+1}}^{2}+\norm{\omega(\cdot,t)}_{V^{m}}^{2} \Bigr\} \leq M_{m}(T) \Bigl\{ \norm{u_{0}}_{V^{m+1}}^{2}+\norm{\omega_{0}}_{V^{m}}^{2} \Bigr\},
\end{displaymath}
where $M_{m}(T)$ is a constant depending on $\norm{u_{0}}_{V^{\min(m,2)}},\ \norm{\omega_{0}}_{V^{\min(m,2)}}$, and
\begin{displaymath}
  M_{0}(T) := \exp \biggl\{ \int_{0}^{T} \norm{H\omega(\cdot,t)}_{L^{\infty}}\,dt \biggr\}.
\end{displaymath}
Finally, Theorem \ref{thm_1d_c_dep} can be proved using a standard regularization technique. The details of these proofs are given in Section \ref{sec_wellp}.

\section{The Finite-Time Blowup of the 1D Model}\label{sec_blowup}
To study the finite-time blowup of the 1D model \eqref{eqn_eat_1d}, it is convenient to establish the following

\begin{theorem}[Beale-Kato-Majda type criterion]
Suppose that
\begin{displaymath}
  u_{0} \in V^{m+1}(S),\qquad \omega_{0} \in V^{m}(S),
\end{displaymath}
and that the solution of \eqref{eqn_eat_1d} in class $CV^{m}$ exists on $[0,T)$. Then the solution cannot be continued in class $CV^{m}$ up to and beyond $T$ if and only if
\begin{equation}
  \int_{0}^{T} \norm{H\omega(\cdot,t)}_{L^{\infty}}\,dt = \infty.
  \label{eqn_1d_bkm}
\end{equation}
\label{thm_1d_bkm}
\end{theorem}
This criterion is similar to Theorem 3.2 proved in \citet{osw2008} and is an analogue of the celebrated Beale-Kato-Majda theorem \citep{bkm1984}. Its proof is given in Section \ref{ssec_wp_bkm}.

We shall now argue that the 1D model \eqref{eqn_eat_1d} develops a singularity in finite time, for the particular initial data
\begin{equation}
  u_{0}(z) = a \sin^{2}(\mu z),\quad a > 0,\qquad \omega_{0}(z) = 0.
  \label{eqn_eat_1d_ic}
\end{equation}
More specifically, we shall show that the velocity gradient
\begin{displaymath}
  v_{z}(0) = H\omega(0) = -\frac{1}{L} \int_{0}^{L} \omega(z) \cot(\mu z)\,dz
\end{displaymath}
at $z = 0$ satisfies a lower bound estimate
\begin{equation}
  {-}v_{z}(0,t) = \abs{v_{z}(0,t)} \geq 2 c_{0} \tan(\tfrac{1}{2} c_{0} t),\qquad c_{0} = \biggl[ \frac{\mu}{L} \int_{0}^{L} u_{0}(z) \cot^{2}(\mu z)\,dz \biggr]^{1/2} = (\tfrac{1}{2} a\mu)^{1/2}.
  \label{eqn_vz_0}
\end{equation}
The finite-time blowup of \eqref{eqn_eat_1d} is a consequence of \eqref{eqn_vz_0} in view of Theorem \ref{thm_1d_bkm}. Note that, for the given initial data, the solution has the property that $u$ is even and has a \emph{double zero} at $z = 0,\
\frac{1}{2} L$, and $\omega,\ v$ are odd at $z = 0,\ \frac{1}{2} L$. In addition, $u,\, u_{z},\, \omega > 0$ and $v < 0$ on $(0,\frac{1}{2} L)$ for all $t > 0$ (for the proof of the last assertion, see \eqref{eqn_v}). These symmetry and sign-preserving
properties mimic the behavior of the solutions of the 3D Euler equations \eqref{eqn_eat} on the wall $r = 1$. In particular, they create a compression flow near $z = 0$ with $v < 0$ for $z > 0$ and $v > 0$ for $z < 0$, completely similar to the scenario
observed in 3D \citep{lh2013}. This provides an intuitive explanation for the finite-time blowup of the 1D model.

The proof of \eqref{eqn_vz_0} proceeds as follows. First, we multiply the $\omega$-equation \eqref{eqn_eat_1d_w} by $\cot(\mu z)/L$ and integrate the resulting equation on $[0,L]$; this yields
\begin{subequations}\label{eqn_ie_w}
\begin{equation}
  {-}v_{zt}(0,t) + I = \frac{\mu}{L} \int_{0}^{L} u(z) \csc^{2}(\mu z)\,dz,
  \label{eqn_ie_w_t}
\end{equation}
where
\begin{align}
  I & = \frac{1}{L} \int_{0}^{L} v(z) \omega_{z}(z) \cot(\mu z)\,dz \nonumber \\
  & = -\frac{1}{L} \int_{0}^{L} \omega(z) \bigl[ v_{z}(z) \cot(\mu z) - \mu v(z) \csc^{2}(\mu z) \bigr]\,dz \nonumber \\
  & = H(\omega v_{z})(0) + \frac{\mu}{L} \int_{0}^{L} \omega(z) v(z) \csc^{2}(\mu z)\,dz =: \frac{1}{2}\, (v_{z})^{2}(0) - I_{1}. \label{eqn_ie_w_I}
\end{align}
\end{subequations}
A direct computation using the definition of $v$ shows
\begin{align}
  v(z) & = \frac{1}{\pi} \biggl[ \int_{0}^{L/2} + \int_{L/2}^{L} \biggr] \omega(y) \log \bigl| \sin[\mu(z-y)] \bigr|\,dy \nonumber \\
  & = \frac{1}{\pi} \int_{0}^{L/2} \omega(y) \Bigl\{ \log \bigl| \sin[\mu(z-y)] \bigr| - \log \bigl| \sin[\mu(z+y)] \bigr| \Bigr\}\,dy \nonumber \\
  & = \frac{1}{\pi} \int_{0}^{L/2} \omega(y) \log \biggl| \frac{\tan(\mu z) - \tan(\mu y)}{\tan(\mu z) + \tan(\mu y)} \biggr|\,dy < 0,\qquad \forall z \in (0,\tfrac{1}{2} L). \label{eqn_v}
\end{align}
Substituting \eqref{eqn_v} into the definition of $I_{1}$ (see \eqref{eqn_ie_w_I}), we deduce
\begin{align*}
  I_{1} & = -\frac{\mu}{L} \int_{0}^{L} \omega(z) v(z) \csc^{2}(\mu z)\,dz \geq -\frac{2\mu}{L} \int_{0}^{L/2} \omega(z) v(z) \cot^{2}(\mu z)\,dz \\
  & = -\frac{2\mu}{\pi L} \int_{0}^{L/2} F(z) \int_{0}^{L/2} F(y) K(y,z)\,dy\,dz,
\end{align*}
where $F(z) = \omega(z) \cot(\mu z)$ and
\begin{displaymath}
  K(y,z) = -w \log \biggl| \frac{w+1}{w-1} \biggr|,\qquad w = \frac{\tan(\mu y)}{\tan(\mu z)}.
\end{displaymath}
By introducing the decomposition
\begin{align*}
  I_{11} & = \frac{2\mu}{\pi L} \int_{0}^{L/2} F(z) \int_{0}^{L/2} F(y)\,dy\,dz, \\
  I_{12} & = -\frac{2\mu}{\pi L} \int_{0}^{L/2} F(z) \int_{0}^{L/2} F(y) \bigl[ K(y,z) + 1 \bigr]\,dy\,dz,
\end{align*}
we write $I_{1} \geq I_{11} + I_{12}$ and compute
\begin{displaymath}
  I_{11} = \frac{1}{2} \biggl[ \frac{2}{L} \int_{0}^{L/2} \omega(z) \cot(\mu z)\,dz \biggr]^{2} = \frac{1}{2}\, (v_{z})^{2}(0).
\end{displaymath}
To estimate $I_{12}$, we introduce another decomposition
\begin{displaymath}
  I_{12} = -\frac{2\mu}{\pi L} \int_{0}^{L/2} F(z) \biggl[ \int_{0}^{z} + \int_{z}^{L/2} \biggr] (\cdots)\,dy\,dz =: J_{11} + J_{12},
\end{displaymath}
and rearrange $J_{12}$ using Fubini's theorem:
\begin{align*}
  J_{12} & = -\frac{2\mu}{\pi L} \int_{0}^{L/2} F(y) \int_{0}^{y} F(z) \bigl[ K(y,z) + 1 \bigr]\,dz\,dy \\
  & = -\frac{2\mu}{\pi L} \int_{0}^{L/2} F(z) \int_{0}^{z} F(y) \bigl[ K(z,y) + 1 \bigr]\,dy\,dz.
\end{align*}
This yields
\begin{displaymath}
  I_{12} = -\frac{2\mu}{\pi L} \int_{0}^{L/2} F(z) \int_{0}^{z} F(y) \bigl[ K(y,z) + K(z,y) + 2 \bigr]\,dy\,dz.
\end{displaymath}
Since
\begin{displaymath}
  K(y,z) + K(z,y) + 2 = -w \log \biggl| \frac{w+1}{w-1} \biggr| - \frac{1}{w} \log \biggl| \frac{w+1}{w-1} \biggr| + 2 \leq 0,\qquad \forall w \geq 0,
\end{displaymath}
and $F \geq 0$ on $[0,\frac{1}{2} L]$, we conclude that $I_{12} \geq 0$ and hence
\begin{displaymath}
  I_{1} \geq I_{11} = \frac{1}{2}\, (v_{z})^{2}(0).
\end{displaymath}
This, combined with \eqref{eqn_ie_w}, leads to the estimate:
\begin{equation}
  {-}v_{zt}(0,t) \geq \frac{\mu}{L} \int_{0}^{L} u(z) \cot^{2}(\mu z)\,dz.
  \label{eqn_lb_w}
\end{equation}

Next, we multiply the $u$-equation \eqref{eqn_eat_1d_u} by $\mu \cot^{2}(\mu z)/L$ and integrate the resulting equation on $[0,L]$; this yields
\begin{subequations}\label{eqn_ie_u}
\begin{equation}
  \frac{d}{dt} \biggl[ \frac{\mu}{L} \int_{0}^{L} u(z,t) \cot^{2}(\mu z)\,dz \biggr] - I_{2} = 0,
  \label{eqn_ie_u_t}
\end{equation}
where
\begin{align}
  I_{2} & = -\frac{\mu}{L} \int_{0}^{L} v(z) u_{z}(z) \cot^{2}(\mu z)\,dz \nonumber \\
  & = -\frac{2\mu}{\pi L} \int_{0}^{L/2} G(z) \int_{0}^{L/2} F(y) K(y,z)\,dy\,dz, \label{eqn_ie_u_I}
\end{align}
\end{subequations}
where $G(z) = u_{z}(z) \cot(\mu z)$ and $F(y),\ K(y,z)$ are given as before. By introducing the decomposition
\begin{align*}
  I_{21} & = \frac{2\mu}{\pi L} \int_{0}^{L/2} G(z) \int_{0}^{L/2} F(y)\,dy\,dz, \\
  I_{22} & = -\frac{2\mu}{\pi L} \int_{0}^{L/2} G(z) \int_{0}^{L/2} F(y) \bigl[ K(y,z) + 1 \bigr]\,dy\,dz,
\end{align*}
we write $I_{2} = I_{21} + I_{22}$ and compute
\begin{displaymath}
  I_{21} \geq -\frac{1}{2}\, v_{z}(0) \biggl[ \frac{\mu}{L} \int_{0}^{L} u(z) \cot^{2}(\mu z)\,dz \biggr].
\end{displaymath}
To estimate $I_{22}$, we introduce another decomposition
\begin{displaymath}
  I_{22} = -\frac{2\mu}{\pi L} \int_{0}^{L/2} G(z) \biggl[ \int_{0}^{z} + \int_{z}^{L/2} \biggr] (\cdots)\,dy\,dz =: J_{21} + J_{22},
\end{displaymath}
where
\begin{align*}
  J_{21} & = -\frac{2\mu}{\pi L} \int_{0}^{L/2} G(z) \int_{0}^{z} F(y) \bigl[ K(y,z) + 1 \bigr]\,dy\,dz, \\
  J_{22} & = -\frac{2\mu}{\pi L} \int_{0}^{L/2} F(z) \int_{0}^{z} G(y) \bigl[ K(z,y) + 1 \bigr]\,dy\,dz.
\end{align*}
Since
\begin{align*}
  K(y,z) & = -w \log \biggl| \frac{w+1}{w-1} \biggr| \leq 0,\qquad\quad w \in [0,1), \\
  K(z,y) & = -\frac{1}{w} \log \biggl| \frac{w+1}{w-1} \biggr| \leq -2,\qquad w \in [0,1),
\end{align*}
and $F,\, G \geq 0$ on $[0,\frac{1}{2} L]$, we conclude that
\begin{displaymath}
  I_{22} \geq \frac{2\mu}{\pi L} \int_{0}^{L/2} \cot(\mu z) \int_{0}^{z} \cot(\mu y) D(y,z)\,dy\,dz,
\end{displaymath}
where $D(y,z) = \omega(z) u_{y}(y) - u_{z}(z) \omega(y)$. \emph{Assuming} for now that
\begin{equation}
  D(y,z) \geq 0,\qquad \text{for all $0 \leq y \leq z \leq \tfrac{1}{2} L$}.
  \label{eqn_D_pos}
\end{equation}
Then
\begin{displaymath}
  I_{2} \geq -\frac{1}{2}\, v_{z}(0) \biggl[ \frac{\mu}{L} \int_{0}^{L} u(z) \cot^{2}(\mu z)\,dz \biggr],
\end{displaymath}
and estimates \eqref{eqn_lb_w}--\eqref{eqn_ie_u} reduce to
\begin{subequations}\label{eqn_1d_dineq}
\begin{align}
  h_{1}'(t) & \geq h_{2}(t), \label{eqn_1d_dineq_1} \\
  h_{2}'(t) & \geq \frac{1}{2}\, h_{1}(t) h_{2}(t), \label{eqn_1d_dineq_2}
\end{align}
\end{subequations}
where
\begin{displaymath}
  h_{1}(t) := -v_{z}(0,t) \geq 0,\qquad h_{2}(t) := \frac{\mu}{L} \int_{0}^{L} u(z,t) \cot^{2}(\mu z)\,dz \geq 0.
\end{displaymath}

The lower bound estimate \eqref{eqn_vz_0} can now be easily derived from \eqref{eqn_1d_dineq}. Indeed, integrating \eqref{eqn_1d_dineq_1} from 0 to $t$ and using the initial condition $h_{1}(0) = 0$, we see
\begin{displaymath}
  h_{1}(t) \geq H(t) := \int_{0}^{t} h_{2}(s)\,ds.
\end{displaymath}
Substituting this estimate into \eqref{eqn_1d_dineq_2} and rearranging, we then deduce
\begin{displaymath}
  H''(t) \geq \frac{1}{2}\, H(t) H'(t),
\end{displaymath}
the repeated integration of which yields
\begin{displaymath}
  H(t) \geq 2 c_{0} \tan(\tfrac{1}{2} c_{0} t),\qquad c_{0} = h_{2}^{1/2}(0).
\end{displaymath}
It follows that $H(t)$, hence $h_{1}(t) = \abs{v_{z}(0,t)}$, blows up no later than $T^{*} = \pi/c_{0}$. For the initial data considered in \eqref{eqn_eat_1d_ic}, we have
\begin{displaymath}
  h_{2}(0) = \frac{\mu}{L} \int_{0}^{L} a\sin^{2}(\mu z) \cot^{2}(\mu z)\,dz = \frac{1}{2}\, a\mu,
\end{displaymath}
so the solution blows up no later than $T^{*} = \sqrt{2\pi L/a}$.

To complete the proof of \eqref{eqn_vz_0} and hence the finite-time blowup of the 1D model \eqref{eqn_eat_1d}, it remains to prove \eqref{eqn_D_pos}. Since $u_{z} > 0$ on $(0,\frac{1}{2} L)$ and
\begin{displaymath}
  \frac{D(y,z)}{u_{z}(z) u_{y}(y)} = \frac{\omega(z)}{u_{z}(z)} - \frac{\omega(y)}{u_{y}(y)} =: Q(z) - Q(y),\qquad Q(z) = \frac{\omega(z)}{u_{z}(z)},
\end{displaymath}
we see that $D(y,z) \geq 0$ for $0 \leq y \leq z \leq \frac{1}{2} L$ if $Q_{z} \geq 0$ on $(0,\frac{1}{2} L)$. Since $Q_{z}$ satisfies the evolution equation
\begin{displaymath}
  Q_{zt} + v Q_{zz} = v_{zz} Q,\qquad Q_{z}(z,0) = 0,
\end{displaymath}
we see that $Q_{z} \geq 0$ on $(0,\frac{1}{2} L)$ if (recall $Q > 0$ on $(0,\frac{1}{2} L)$)
\begin{equation}
  v_{zz}(z) \geq 0,\qquad \forall z \in (0,\tfrac{1}{2} L).
  \label{eqn_vzz}
\end{equation}
We have not been able to prove \eqref{eqn_vzz} rigorously but have verified this condition \emph{numerically} for solutions generated from \eqref{eqn_eat_1d_ic}. This strongly indicates the existence of a finite-time singularity for the 1D model
\eqref{eqn_eat_1d}.

It is interesting to note that condition \eqref{eqn_vzz} can be interpreted from a geometric point of view: if the flow field $v$ is odd at $z = 0,\ \frac{1}{2} L$ and is convex on $(0,\frac{1}{2} L)$, it will necessarily be negative for $z > 0$ and
positive for $z < 0$, creating a compression flow near $z = 0$. If the convexity of the velocity field is preserved by the flow, then the compression mechanism near $z = 0$ will be sustained and reinforced, eventually leading to the formation of a
finite-time singularity.

\section{Proof of the Well-Posedness}\label{sec_wellp}
In this section we give the proofs of Theorems \ref{thm_1d_loc_ex}--\ref{thm_1d_c_dep} and \ref{thm_1d_bkm}.

\subsection{An Abstract Existence Theorem}\label{ssec_wp_abs_evol}
The local existence of solutions of \eqref{eqn_eat_1d} can be proved using various techniques such as successive approximation, fixed-point theorem, or Galerkin approximation. In what follows, we shall make use of an abstract existence theorem of
\citet{kl1984}, which is based on a variant of Galerkin approximation. To state the theorem, we consider abstract nonlinear evolution equations of the form
\begin{equation}
  u_{t} + A(t,u) = 0,\qquad t \geq 0,\ u(0) = u_{0},
  \label{eqn_abs_evol}
\end{equation}
where $A$ is a nonlinear operator. To define $A$ precisely, we introduce the notion of \emph{admissible triplet}, which consists of three real separable Banach spaces $\{Y,H,X\}$ with the properties:
\begin{enumerate}
  \item $Y \subset H \subset X$, with the inclusions continuous and dense;
  \item $H$ is a Hilbert space, with inner product $(\cdot,\cdot)_{H}$ and norm $\norm{\cdot}_{H} = (\cdot,\cdot)_{H}^{1/2}$;
  \item there is a continuous, nondegenerate bilinear form on $Y \times X$, denoted by $\langle \cdot,\cdot \rangle$, such that
  \begin{displaymath}
    \langle v,u \rangle = (v,u)_{H},\qquad \forall v \in Y,\ u \in H.
  \end{displaymath}
\end{enumerate}

With these notions, the existence theorem of \citet{kl1984} reads
\begin{theorem}[Abstract existence theorem]
Let $\{Y,H,X\}$ be an admissible triplet. Let $A$ be a sequentially weakly continuous map on $[0,T_{0}] \times H$ into $X$ such that
\begin{equation}
  \langle v,A(t,v) \rangle \geq -\beta(\norm{v}_{H}^{2}),\qquad \forall t \in [0,T_{0}],\ v \in Y,
  \label{eqn_abs_lb}
\end{equation}
where $\beta(r) \geq 0$ is a monotone increasing function of $r \geq 0$. Then for any $u_{0} \in H$, there exists $T \in (0,T_{0}]$ and a solution $u$ of \eqref{eqn_abs_evol} in the class
\begin{displaymath}
  u \in C_{w}([0,T]; H) \cap C_{w}^{1}([0,T]; X),
\end{displaymath}
where the subscript $w$ in $C_{w}$ and $C_{w}^{1}$ indicates weak continuity. Moreover, one has
\begin{displaymath}
  \norm{u(t)}_{H}^{2} \leq \rho(t),\qquad t \in [0,T],
\end{displaymath}
where $\rho$ solves the scalar differential equation
\begin{equation}
  \rho'(t) = 2\beta(\rho),\qquad \rho(0) = \norm{u_{0}}_{H}^{2}.
  \label{eqn_abs_maj}
\end{equation}
If the solution to \eqref{eqn_abs_maj} is not unique, $\rho$ is understood as the maximal solution.
\label{thm_abs_evol}
\end{theorem}

Theorem \ref{thm_abs_evol} is not concerned with the uniqueness of the solution, neither is it concerned with the existence of strongly continuous solutions. However, both issues can be settled in a straightforward manner in our case.

With the aid of Theorem \ref{thm_abs_evol}, we shall prove the local existence part of Theorem \ref{thm_1d_loc_ex} by introducing, for any $k \geq 0$, the tensor product space
\begin{displaymath}
  W^{k}(S) = V^{k+1}(S) \times V^{k}(S).
\end{displaymath}
We equip $W^{k}(S)$ with the (natural) inner product
\begin{displaymath}
  (f,g)_{W^{k}} = (f_{1},g_{1})_{V^{k+1}} + (f_{2},g_{2})_{V^{k}},\qquad f,\, g \in W^{k}(S),
\end{displaymath}
and the norm
\begin{displaymath}
  \norm{f}_{W^{k}} = (f,f)_{W^{k}}^{1/2},\qquad f \in W^{k}(S).
\end{displaymath}
In addition, we define for any $m \geq 1$ the triplet
\begin{displaymath}
  Y = W^{2m}(S),\qquad H = W^{m}(S),\qquad X = W^{0}(S),
\end{displaymath}
and the continuous bilinear form $\langle \cdot,\cdot \rangle\colon Y \times X \to \mathbb{R}$:
\begin{displaymath}
  \langle f,g \rangle = (-1)^{m} \int_{0}^{L} \partial_{z}^{2m+1} f_{1} \cdot \partial_{z} g_{1}\,dz + (-1)^{m} \int_{0}^{L} \partial_{z}^{2m} f_{2} \cdot g_{2}\,dz,\qquad f \in Y,\ g \in X.
\end{displaymath}
Integration by parts shows that
\begin{displaymath}
  \langle f,g \rangle = (f_{1},g_{1})_{V^{m+1}} + (f_{2},g_{2})_{V^{m}} = (f,g)_{H},\qquad f \in Y,\ g \in H,
\end{displaymath}
so $\{Y,H,X\}$ forms an admissible triplet. Finally, we introduce the nonlinear mapping
\begin{equation}
  A(h) = (v u_{z},v \omega_{z} - u_{z}),\qquad h = (u,\omega) \in W^{m}(S),\ v \in V^{m+1}(S) \text{ with } v_{z} = H\omega.
  \label{eqn_1d_a}
\end{equation}

To apply Theorem \ref{thm_abs_evol}, we need to show that $A$ defines a map from $H$ into $X$, that $A$ is sequentially weakly continuous, and that $\langle h,A(h) \rangle$ satisfies the estimate \eqref{eqn_abs_lb}. The proof of these facts relies on two
basic estimates of the operator $A$, which we shall derive in the next section.

\subsection{Basic Estimates}\label{ssec_wp_est}
The basic tool that we shall need is the following
\begin{prop}[Calculus inequalities]
Let $k \geq 0$ be any nonnegative integer and let $f,\, g \in L^{\infty}(S) \cap V^{k}(S)$. Then
\begin{subequations}\label{eqn_cal_ineq_0}
\begin{gather}
  \norm{fg}_{V^{k}} \leq C \Bigl\{ \norm{f}_{L^{\infty}} \norm{g}_{V^{k}} + (1-\delta_{k,0}) \norm{f}_{V^{k}} \norm{g}_{L^{\infty}} \Bigr\}, \label{eqn_cal_ineq_0_fg} \\
  \intertext{where $C$ is an absolute constant depending only on $k$ and $L$ and $\delta_{k,0}$ is the usual Kronecker delta symbol, with value 1 at $k = 0$ and 0 otherwise. If $k \geq 1$, then there also holds}
  \norm{\partial_{z}^{k} (fg) - f \partial_{z}^{k} g}_{V^{0}} \leq C \Bigl\{ \norm{f_{z}}_{L^{\infty}} \norm{g}_{V^{k-1}} + (1-\delta_{k,1}) \norm{f}_{V^{k}} \norm{g}_{L^{\infty}} \Bigr\}. \label{eqn_cal_ineq_0_fgd}
\end{gather}
\end{subequations}
\label{prop_cal_ineq_0}
\end{prop}

Inequalities \eqref{eqn_cal_ineq_0} are well known and hold true more generally for functions $f,\, g \in L^{\infty} \cap H^{k}$. For a proof of these results, see for example \citet{mb2002}. It is also worth noting that \eqref{eqn_cal_ineq_0_fgd} holds
true for $g \in L^{\infty} \cap H^{k-1}$ by the usual density argument, even if the individual terms on the left-hand side may not belong to $H^{0}$.

Besides Proposition \ref{prop_cal_ineq_0}, the following well-known facts will also be used in the sequel.
\begin{subequations}\label{eqn_cal_fact}
\begin{enumerate}
  \item The Hilbert transform $H$ is an isometry on $V^{k}(S)$ for all $k \geq 0$, with $\norm{Hf}_{V^{k}} = \norm{f}_{V^{k}}$. In addition, $H$ commutes with $\partial_{z}$, i.e. $H(f_{z}) = (Hf)_{z}$.
  \item The Poincar\'{e} inequality asserts that $V^{k}(S) \subset V^{j}(S)$ for all $k > j \geq 0$, with
  \begin{equation}
    \norm{f}_{V^{j}} \leq c_{0}^{k-j} \norm{f}_{V^{k}},\qquad f \in V^{k}(S),\ k > j \geq 0.
    \label{eqn_cal_fact_p}
  \end{equation}
  The constant $c_{0}$ in the above inequality can be computed explicitly, with $c_{0} = L/(2\pi)$.
  \item The Sobolev imbedding theorem asserts that $V^{k}(S) \subset L^{\infty}(S)$ for all $k \geq 1$, with
  \begin{equation}
    \norm{f}_{L^{\infty}} \leq \tilde{c}_{0} \norm{f_{z}}_{V^{0}} \leq \tilde{c}_{0} c_{0}^{k-1} \norm{f}_{V^{k}},\qquad f \in V^{k}(S),\ k \geq 1.
    \label{eqn_cal_fact_s}
  \end{equation}
  The constant $\tilde{c}_{0}$ in the above inequality can be computed explicitly, with $\tilde{c}_{0} = L/(2\sqrt{3})$.
  \item As a result of the Sobolev imbedding theorem and Proposition \ref{prop_cal_ineq_0}, $V^{k}(S)$ is a Banach algebra for all $k \geq 1$, with
  \begin{equation}
    \norm{fg}_{V^{k}} \leq C \norm{f}_{V^{k}} \norm{g}_{V^{k}},\qquad f,\, g \in V^{k}(S),\ k \geq 1,
    \label{eqn_cal_fact_b}
  \end{equation}
  where $C$ is an absolute constant depending only on $k$ and $L$.
\end{enumerate}
\end{subequations}

With the aid of these tools, we shall prove two basic estimates for the nonlinear operator $A$ defined in \eqref{eqn_1d_a}. The first estimate concerns the boundedness and (strong) continuity of $A$.
\begin{lemma}
Let $k \geq 1$ be any positive integer. Let $u \in V^{k+1}(S),\ \omega,\, \tilde{\omega} \in V^{k}(S)$ and $\tilde{v} \in V^{k+1}(S)$ be such that $\tilde{v}_{z} = H\tilde{\omega}$. Then $\tilde{v} u_{z} \in V^{k}(S),\ \tilde{v} \omega_{z} \in
V^{k-1}(S)$, and
\begin{subequations}\label{eqn_cal_ineq_1}
\begin{align}
  \norm{\tilde{v} u_{z}}_{V^{k}} & \leq C \Bigl\{ \norm{\tilde{\omega}}_{V^{0}} \norm{u}_{V^{k+1}} + \norm{u}_{V^{2}} \norm{\tilde{\omega}}_{V^{k-1}} \Bigr\} \leq C \norm{\tilde{\omega}}_{V^{k-1}} \norm{u}_{V^{k+1}}, \label{eqn_cal_ineq_1_u} \\
  \norm{\tilde{v} \omega_{z}}_{V^{k-1}} & \leq C \Bigl\{ \norm{\tilde{\omega}}_{V^{0}} \norm{\omega}_{V^{k}} + (1-\delta_{k,1}) \norm{\omega}_{V^{2}} \norm{\tilde{\omega}}_{V^{k-2}} \Bigr\} \leq C \norm{\tilde{\omega}}_{V^{k-1}} \norm{\omega}_{V^{k}}, \label{eqn_cal_ineq_1_w}
\end{align}
\end{subequations}
where $C$ is an absolute constant depending only on $k$ and $L$. \label{lmm_cal_ineq_1}
\end{lemma}

\begin{proof}
\eqref{eqn_cal_ineq_1_u} is a direct consequence of the calculus inequality \eqref{eqn_cal_ineq_0_fg}:
\begin{displaymath}
  \norm{\tilde{v} u_{z}}_{V^{k}} \leq C \Bigl\{ \norm{\tilde{v}}_{L^{\infty}} \norm{u_{z}}_{V^{k}} + \norm{u_{z}}_{L^{\infty}} \norm{\tilde{v}}_{V^{k}} \Bigr\},
\end{displaymath}
and the isometry property of the Hilbert transform:
\begin{displaymath}
  \norm{\tilde{v}}_{V^{k}} = \norm{\tilde{v}_{z}}_{V^{k-1}} = \norm{H\tilde{\omega}}_{V^{k-1}} = \norm{\tilde{\omega}}_{V^{k-1}}.
\end{displaymath}
Combining these estimates, using Sobolev's imbedding theorem and noting that $k \geq 1$, we obtain
\begin{align*}
  \norm{\tilde{v} u_{z}}_{V^{k}} & \leq C \Bigl\{ \norm{\tilde{v}_{z}}_{V^{0}} \norm{u}_{V^{k+1}} + \norm{u_{zz}}_{V^{0}} \norm{\tilde{\omega}}_{V^{k-1}} \Bigr\} \\
  & \leq C \Bigl\{ \norm{\tilde{\omega}}_{V^{0}} \norm{u}_{V^{k+1}} + \norm{u}_{V^{2}} \norm{\tilde{\omega}}_{V^{k-1}} \Bigr\} \leq C \norm{\tilde{\omega}}_{V^{k-1}} \norm{u}_{V^{k+1}},
\end{align*}
which is \eqref{eqn_cal_ineq_1_u}. To prove \eqref{eqn_cal_ineq_1_w}, we follow the same steps, utilizing the calculus inequality \eqref{eqn_cal_ineq_0_fg}:
\begin{displaymath}
  \norm{\tilde{v} \omega_{z}}_{V^{k-1}} \leq C \Bigl\{ \norm{\tilde{v}}_{L^{\infty}} \norm{\omega_{z}}_{V^{k-1}} + (1-\delta_{k,1}) \norm{\omega_{z}}_{L^{\infty}} \norm{\tilde{v}}_{V^{k-1}} \Bigr\},
\end{displaymath}
and the isometry property of the Hilbert transform (for $k \geq 2$):
\begin{displaymath}
  \norm{\tilde{v}}_{V^{k-1}} = \norm{\tilde{v}_{z}}_{V^{k-2}} = \norm{H\tilde{\omega}}_{V^{k-2}} = \norm{\tilde{\omega}}_{V^{k-2}}.
\end{displaymath}
Combining these estimates and using Sobolev's imbedding theorem then yields
\begin{align*}
  \norm{\tilde{v} \omega_{z}}_{V^{k-1}} & \leq C \Bigl\{ \norm{\tilde{v}_{z}}_{V^{0}} \norm{\omega}_{V^{k}} + (1-\delta_{k,1}) \norm{\omega_{zz}}_{V^{0}} \norm{\tilde{\omega}}_{V^{k-2}} \Bigr\} \\
  & \leq C \Bigl\{ \norm{\tilde{\omega}}_{V^{0}} \norm{\omega}_{V^{k}} + (1-\delta_{k,1}) \norm{\omega}_{V^{2}} \norm{\tilde{\omega}}_{V^{k-2}} \Bigr\} \leq C \norm{\tilde{\omega}}_{V^{k-1}} \norm{\omega}_{V^{k}},
\end{align*}
which is \eqref{eqn_cal_ineq_1_w}.
\end{proof}

As an immediate consequence of Lemma \ref{lmm_cal_ineq_1}, we have the following
\begin{prop}
Let $m \geq 1$ be any positive integer. The nonlinear operator $A$ defined by
\begin{displaymath}
  A(h) = (v u_{z},v \omega_{z} - u_{z}),\qquad h = (u,\omega) \in W^{m}(S),\ v \in V^{m+1}(S) \text{ with } v_{z} = H\omega,
\end{displaymath}
maps $W^{m}(S)$ continuously (in the strong topology) into $W^{k}(S)$ for all $0 \leq k \leq m-1$. In particular, we have
\begin{align*}
  \norm{A(h)}_{W^{k}} & \leq C \Bigl\{ \norm{h}_{W^{m}} + 1 \Bigr\} \norm{h}_{W^{m}},& & \forall h \in W^{m}(S), \\
  \norm{A(h_{1}) - A(h_{2})}_{W^{k}} & \leq C \Bigl\{ \norm{h_{1}}_{W^{m}} + \norm{h_{2}}_{W^{m}} + 1 \Bigr\} \norm{h_{1} - h_{2}}_{W^{m}},& & \forall h_{1},\, h_{2} \in W^{m}(S),
\end{align*}
where $C$ is an absolute constant depending only on $m$ and $L$. \label{prop_cal_ineq_1}
\end{prop}

\begin{proof}
Using Poincar\'{e}'s inequality and Lemma \ref{lmm_cal_ineq_1}, we have, for each $0 \leq k \leq m-1$,
\begin{align*}
  \norm{v u_{z}}_{V^{k+1}} & \leq C \norm{v u_{z}}_{V^{m}} \leq C \norm{\omega}_{V^{m}} \norm{u}_{V^{m+1}}, \\
  \norm{v \omega_{z} - u_{z}}_{V^{k}} & \leq C \norm{v \omega_{z} - u_{z}}_{V^{m-1}} \leq C \Bigl\{ \norm{\omega}_{V^{m}}^{2} + \norm{u}_{V^{m+1}} \Bigr\}.
\end{align*}
Hence
\begin{displaymath}
  \norm{A(h)}_{W^{k}} \leq C \Bigl\{ \norm{h}_{W^{m}} + 1 \Bigr\} \norm{h}_{W^{m}},\qquad \forall h \in W^{m}(S),
\end{displaymath}
which shows that $A$ maps $W^{m}(S)$ into $W^{k}(S)$. In addition, for any $h_{1} = (u_{1},\omega_{1}),\ h_{2} = (u_{2},\omega_{2}) \in W^{m}(S)$, we have
\begin{align*}
  \norm{v_{1} u_{1,z} - v_{2} u_{2,z}}_{V^{m}} & \leq \norm{\tilde{v} u_{1,z}}_{V^{m}} + \norm{v_{2} \tilde{u}_{z}}_{V^{m}}, \\
  \norm{v_{1} \omega_{1,z} - v_{2} \omega_{2,z}}_{V^{m-1}} & \leq \norm{\tilde{v} \omega_{1,z}}_{V^{m-1}} + \norm{v_{2} \tilde{\omega}_{z}}_{V^{m-1}},
\end{align*}
where $(\tilde{u},\tilde{\omega}) = (u_{1}-u_{2},\omega_{1}-\omega_{2})$ and $\tilde{v} = v_{1}-v_{2}$. Hence another application of Lemma \ref{lmm_cal_ineq_1} yields
\begin{align*}
  \norm{v_{1} u_{1,z} - v_{2} u_{2,z}}_{V^{k+1}} & \leq C \Bigl\{ \norm{u_{1}}_{V^{m+1}} \norm{\tilde{\omega}}_{V^{m}} + \norm{\omega_{2}}_{V^{m}} \norm{\tilde{u}}_{V^{m+1}} \Bigr\}, \\
  \norm{v_{1} \omega_{1,z} - v_{2} \omega_{2,z}}_{V^{k}} & \leq C \Bigl\{ \norm{\omega_{1}}_{V^{m}} + \norm{\omega_{2}}_{V^{m}} \Bigr\} \norm{\tilde{\omega}}_{V^{m}},
\end{align*}
from which we deduce that
\begin{displaymath}
  \norm{A(h_{1}) - A(h_{2})}_{W^{k}} \leq C \Bigl\{ \norm{h_{1}}_{W^{m}} + \norm{h_{2}}_{W^{m}} + 1 \Bigr\} \norm{h_{1} - h_{2}}_{W^{m}},\qquad \forall h_{1},\, h_{2} \in W^{m}(S).
\end{displaymath}
This shows that $A$ is strongly continuous from $W^{m}(S)$ to $W^{k}(S)$.
\end{proof}

In particular, Proposition \ref{prop_cal_ineq_1} shows that $A$ maps $H = W^{m}(S)$ continuously into $X = W^{0}(S)$.

The second estimate we shall prove for the operator $A$ concerns the semi-boundedness of the nonlinear pairing $\langle h,A(h) \rangle$. Note that by Proposition \ref{prop_cal_ineq_1} and Poincar\'{e}'s inequality,
\begin{displaymath}
  A(h) \in W^{2m-1}(S) \subset H,\qquad \forall h \in Y = W^{2m}(S),\ m \geq 1,
\end{displaymath}
so to study $\langle h,A(h) \rangle$ it suffices to consider $(h,A(h))_{H}$.

\begin{lemma}
Let $k \geq 1$ be any positive integer. Let $u,\, \tilde{u} \in V^{k+1}(S),\ \omega,\, \tilde{\omega} \in V^{k}(S)$ and $v,\, \tilde{v} \in V^{k+1}(S)$ be such that $v_{z} = H\omega,\ \tilde{v}_{z} = H\tilde{\omega}$. Then
\begin{subequations}\label{eqn_cal_ineq_2}
\begin{align}
  \abs{(\tilde{u}, \tilde{v} u_{z})_{V^{k}}} & \leq C \norm{\tilde{u}}_{V^{k}} \Bigl\{ \norm{\tilde{\omega}}_{V^{0}} \norm{u}_{V^{k+1}} + \norm{u}_{V^{2}} \norm{\tilde{\omega}}_{V^{k-1}} \Bigr\}, \label{eqn_cal_ineq_2_utu} \\
  \abs{(\tilde{\omega}, \tilde{v} \omega_{z})_{V^{k-1}}} & \leq C \norm{\tilde{\omega}}_{V^{k-1}} \Bigl\{ \norm{\tilde{\omega}}_{V^{0}} \norm{\omega}_{V^{k}} + (1-\delta_{k,1}) \norm{\omega}_{V^{2}} \norm{\tilde{\omega}}_{V^{k-2}} \Bigr\}, \label{eqn_cal_ineq_2_wtw} \\
  \abs{(\tilde{u}, v \tilde{u}_{z})_{V^{k}}} & \leq C \norm{\tilde{u}}_{V^{k}} \Bigl\{ \norm{H\omega}_{L^{\infty}} \norm{\tilde{u}}_{V^{k}} + (1-\delta_{k,1}) \norm{\tilde{u}_{z}}_{L^{\infty}} \norm{\omega}_{V^{k-1}} \Bigr\}, \label{eqn_cal_ineq_2_tu2} \\
  \abs{(\tilde{\omega}, v \tilde{\omega}_{z})_{V^{k-1}}} & \leq C \norm{\tilde{\omega}}_{V^{k-1}} \Bigl\{ \norm{H\omega}_{L^{\infty}} \norm{\tilde{\omega}}_{V^{k-1}}
  + (1-\delta_{k,1}) (1-\delta_{k,2}) \norm{\tilde{\omega}_{z}}_{L^{\infty}} \norm{\omega}_{V^{k-2}} \Bigr\}, \label{eqn_cal_ineq_2_tw2}
\end{align}
\end{subequations}
where $C$ is an absolute constant depending only on $k$ and $L$. \label{lmm_cal_ineq_2}
\end{lemma}

\begin{proof}
\eqref{eqn_cal_ineq_2_utu} is a direct consequence of the Cauchy-Schwarz inequality and estimate \eqref{eqn_cal_ineq_1_u}:
\begin{align*}
  \abs{(\tilde{u}, \tilde{v} u_{z})_{V^{k}}} & \leq \norm{\tilde{u}}_{V^{k}} \norm{\tilde{v} u_{z}}_{V^{k}} \leq C \norm{\tilde{u}}_{V^{k}} \Bigl\{ \norm{\tilde{\omega}}_{V^{0}} \norm{u}_{V^{k+1}} + \norm{u}_{V^{2}} \norm{\tilde{\omega}}_{V^{k-1}} \Bigr\}.
\end{align*}
Likewise, \eqref{eqn_cal_ineq_2_wtw} follows from the Cauchy-Schwarz inequality and estimate \eqref{eqn_cal_ineq_1_w}:
\begin{align*}
  \abs{(\tilde{\omega}, \tilde{v} \omega_{z})_{V^{k-1}}} & \leq \norm{\tilde{\omega}}_{V^{k-1}} \norm{\tilde{v} \omega_{z}}_{V^{k-1}}
  \leq C \norm{\tilde{\omega}}_{V^{k-1}} \Bigl\{ \norm{\tilde{\omega}}_{V^{0}} \norm{\omega}_{V^{k}} + (1-\delta_{k,1}) \norm{\omega}_{V^{2}} \norm{\tilde{\omega}}_{V^{k-2}} \Bigr\}.
\end{align*}
To prove \eqref{eqn_cal_ineq_2_tu2}, we introduce the decomposition
\begin{displaymath}
  (\tilde{u}, v \tilde{u}_{z})_{V^{k}} = (\partial_{z}^{k} \tilde{u}, \partial_{z}^{k} (v \tilde{u}_{z}) - v \partial_{z}^{k+1} \tilde{u})_{V^{0}} + (\partial_{z}^{k} \tilde{u}, v \partial_{z}^{k+1} \tilde{u})_{V^{0}} =: I_{1} + I_{2}.
\end{displaymath}
The first term $I_{1}$ on the right-hand side can be bounded using the calculus inequality \eqref{eqn_cal_ineq_0_fgd}:
\begin{align*}
  \abs{I_{1}} & \leq \norm{\partial_{z}^{k} \tilde{u}}_{V^{0}} \norm{\partial_{z}^{k} (v \tilde{u}_{z}) - v \partial_{z}^{k+1} \tilde{u}}_{V^{0}} \\
  & \leq C \norm{\tilde{u}}_{V^{k}} \Bigl\{ \norm{v_{z}}_{L^{\infty}} \norm{\tilde{u}}_{V^{k}} + (1-\delta_{k,1}) \norm{v}_{V^{k}} \norm{\tilde{u}_{z}}_{L^{\infty}} \Bigr\}.
\end{align*}
As for $I_{2}$, integration by parts yields
\begin{displaymath}
  I_{2} = \int_{0}^{L} \partial_{z}^{k} \tilde{u} \cdot v \partial_{z}^{k+1} \tilde{u}\,dz = -\frac{1}{2} \int_{0}^{L} v_{z} (\partial_{z}^{k} \tilde{u})^{2}\,dz,
\end{displaymath}
hence
\begin{displaymath}
  \abs{I_{2}} \leq C \norm{v_{z}}_{L^{\infty}} \norm{\tilde{u}}_{V^{k}}^{2}.
\end{displaymath}
In summary,
\begin{displaymath}
  \abs{(\tilde{u}, v \tilde{u}_{z})_{V^{k}}} \leq C \norm{\tilde{u}}_{V^{k}} \Bigl\{ \norm{H\omega}_{L^{\infty}} \norm{\tilde{u}}_{V^{k}} + (1-\delta_{k,1}) \norm{\tilde{u}_{z}}_{L^{\infty}} \norm{\omega}_{V^{k-1}} \Bigr\},
\end{displaymath}
as is to be shown.

Finally, to prove \eqref{eqn_cal_ineq_2_tw2} we write
\begin{displaymath}
  (\tilde{\omega}, v \tilde{\omega}_{z})_{V^{k-1}} = (\partial_{z}^{k-1} \tilde{\omega}, \partial_{z}^{k-1} (v \tilde{\omega}_{z}) - v \partial_{z}^{k} \tilde{\omega})_{V^{0}}
  + (\partial_{z}^{k-1} \tilde{\omega}, v \partial_{z}^{k} \tilde{\omega})_{V^{0}} =: I_{3} + I_{4},
\end{displaymath}
where
\begin{align*}
  k & = 1: & I_{3} & = 0, \\
  k & \geq 2: & \abs{I_{3}} & \leq \norm{\partial_{z}^{k-1} \tilde{\omega}}_{V^{0}} \norm{\partial_{z}^{k-1} (v \tilde{\omega}_{z}) - v \partial_{z}^{k} \tilde{\omega}}_{V^{0}} \\
  & & & \leq C \norm{\tilde{\omega}}_{V^{k-1}} \Bigl\{ \norm{v_{z}}_{L^{\infty}} \norm{\tilde{\omega}}_{V^{k-1}} + (1-\delta_{k,2}) \norm{v}_{V^{k-1}} \norm{\tilde{\omega}_{z}}_{L^{\infty}} \Bigr\},
\end{align*}
and
\begin{displaymath}
  I_{4} = \int_{0}^{L} \partial_{z}^{k-1} \tilde{\omega} \cdot v \partial_{z}^{k} \tilde{\omega}\,dz = -\frac{1}{2} \int_{0}^{L} v_{z} (\partial_{z}^{k-1} \tilde{\omega})^{2}\,dz.
\end{displaymath}
Therefore,
\begin{displaymath}
  \abs{(\tilde{\omega}, v \tilde{\omega}_{z})_{V^{k-1}}} \leq C \norm{\tilde{\omega}}_{V^{k-1}} \Bigl\{ \norm{H\omega}_{L^{\infty}} \norm{\tilde{\omega}}_{V^{k-1}} + (1-\delta_{k,1}) (1-\delta_{k,2}) \norm{\tilde{\omega}_{z}}_{L^{\infty}} \norm{\omega}_{V^{k-2}} \Bigr\},
\end{displaymath}
as is to be shown.
\end{proof}

As an immediate consequence of Lemma \ref{lmm_cal_ineq_2}, we have the following
\begin{prop}
Let $m \geq 1$ be any positive integer. The nonlinear operator $A$ defined by
\begin{displaymath}
  A(h) = (v u_{z},v \omega_{z} - u_{z}),\qquad h = (u,\omega) \in H = W^{m}(S),\ v \in V^{m+1}(S) \text{ with } v_{z} = H\omega,
\end{displaymath}
satisfies the estimate
\begin{displaymath}
  \langle h,A(h) \rangle \geq -\beta(\norm{h}_{H}^{2}),\qquad \forall h \in Y = W^{2m}(S),
\end{displaymath}
where
\begin{displaymath}
  \beta(r) = Cr(1+r^{1/2}),
\end{displaymath}
with $C$ being an absolute constant depending only on $m$ and $L$. \label{prop_cal_ineq_2}
\end{prop}

\begin{proof}
For each $h = (u,\omega) \in Y \subset W^{m+1}(S)$, we have, by Proposition \ref{prop_cal_ineq_1} and Poincar\'{e}'s inequality, $A(h) \in W^{2m-1}(S) \subset H$. Hence
\begin{displaymath}
  \langle h,A(h) \rangle = (h,A(h))_{H} = (u,v u_{z})_{V^{m+1}} + (\omega,v \omega_{z} - u_{z})_{V^{m}}.
\end{displaymath}
Using estimates \eqref{eqn_cal_ineq_2_tu2}--\eqref{eqn_cal_ineq_2_tw2} from Lemma \ref{lmm_cal_ineq_2} with $k = m+1,\ \tilde{u} = u$ and $\tilde{\omega} = \omega$, we have, for $m \geq 1$,
\begin{align*}
  \abs{(u,v u_{z})_{V^{m+1}}} & \leq C \norm{u}_{V^{m+1}} \Bigl\{ \norm{H\omega}_{L^{\infty}} \norm{u}_{V^{m+1}} + \norm{u_{z}}_{L^{\infty}} \norm{\omega}_{V^{m}} \Bigr\} \\
  & \leq C \norm{u}_{V^{m+1}} \Bigl\{ \norm{\omega}_{V^{m}} \norm{u}_{V^{m+1}} + \norm{u}_{V^{m+1}} \norm{\omega}_{V^{m}} \Bigr\} \leq C \norm{\omega}_{V^{m}} \norm{u}_{V^{m+1}}^{2}, \\
  \abs{(\omega,v \omega_{z})_{V^{m}}} & \leq C \norm{\omega}_{V^{m}} \Bigl\{ \norm{H\omega}_{L^{\infty}} \norm{\omega}_{V^{m}} + (1-\delta_{m,1}) \norm{\omega_{z}}_{L^{\infty}} \norm{\omega}_{V^{m-1}} \Bigr\} \\
  & \leq C \norm{\omega}_{V^{m}} \Bigl\{ \norm{\omega}_{V^{m}}^{2} + (1-\delta_{m,1}) \norm{\omega}_{V^{m}}^{2} \Bigr\} \leq C \norm{\omega}_{V^{m}}^{3}.
\end{align*}
This shows that
\begin{displaymath}
  \abs{\langle h,A(h) \rangle} \leq C \Bigl\{ \norm{h}_{H} + 1 \Bigr\} \norm{h}_{H}^{2},\qquad \forall h \in Y,
\end{displaymath}
and hence the proposition follows.
\end{proof}

\subsection{Proof of Theorem \ref{thm_1d_loc_ex}}\label{ssec_wp_loc_ex}
Now we are ready to prove Theorem \ref{thm_1d_loc_ex}. By Proposition \ref{prop_cal_ineq_1} and Proposition \ref{prop_cal_ineq_2}, the conditions of Theorem \ref{thm_abs_evol} are satisfied, hence for any given initial data $h_{0} = (u_{0},\omega_{0})
\in H = W^{m}(S)$, there exists $T > 0$ depending only on $\norm{h_{0}}_{H}^{2} = \norm{u_{0}}_{V^{m+1}}^{2} + \norm{\omega_{0}}_{V^{m}}^{2}$ such that the 1D model \eqref{eqn_eat_1d} has a solution
\begin{displaymath}
  h = (u,\omega) \in C_{w}([0,T]; H) \cap C_{w}^{1}([0,T]; X).
\end{displaymath}
To prove uniqueness, assume $h_{1} = (u_{1},\omega_{1}),\ h_{2} = (u_{2},\omega_{2})$ are two solutions to \eqref{eqn_eat_1d} with the same initial data $h_{0}$. Subtracting the two equations satisfied by $h_{1},\ h_{2}$, taking the $X$-inner product of
the resulting equation with $\tilde{h} = (\tilde{u},\tilde{\omega}) = (u_{1}-u_{2},\omega_{1}-\omega_{2})$, and observing that $\tilde{h} \in \mathrm{Lip}([0,T]; X)$ (which implies $\partial_{t} \norm{\tilde{h}}_{X}^{2} = 2(\tilde{h},\tilde{h}_{t})_{X}$
a.e. $t \in [0,T]$), we obtain
\begin{displaymath}
  \frac{1}{2} \frac{d}{dt} \norm{\tilde{h}}_{X}^{2} = -(\tilde{h}, A(h_{1})-A(h_{2}))_{X} = I_{1} + I_{21} + I_{22},
\end{displaymath}
where
\begin{align*}
  I_{1} & = -(\tilde{u}, v_{1} u_{1,z} - v_{2} u_{2,z})_{V^{1}} = -(\tilde{u}, \tilde{v} u_{1,z})_{V^{1}} - (\tilde{u}, v_{2} \tilde{u}_{z})_{V^{1}},& \tilde{v} & = v_{1}-v_{2}, \\
  I_{21} & = -(\tilde{\omega}, v_{1} \omega_{1,z} - v_{2} \omega_{2,z})_{V^{0}} = -(\tilde{\omega}, \tilde{v} \omega_{1,z})_{V^{0}} - (\tilde{\omega}, v_{2} \tilde{\omega}_{z})_{V^{0}},& I_{22} & = (\tilde{\omega}, \tilde{u}_{z})_{V^{0}}.
\end{align*}
Since $h_{1},\, h_{2} \in H \subset W^{1}(S)$, Lemma \ref{lmm_cal_ineq_2} applies with $k = 1$, yielding
\begin{align*}
  \abs{I_{1}} & \leq C \Bigl\{ \norm{\tilde{\omega}}_{V^{0}} \norm{\tilde{u}}_{V^{1}} \norm{u_{1}}_{V^{2}} + \norm{\omega_{2}}_{V^{1}} \norm{\tilde{u}}_{V^{1}}^{2} \Bigr\}, \\
  \abs{I_{21}} & \leq C \Bigl\{ \norm{\tilde{\omega}}_{V^{0}}^{2} \norm{\omega_{1}}_{V^{1}} + \norm{\omega_{2}}_{V^{1}} \norm{\tilde{\omega}}_{V^{0}}^{2} \Bigr\}.
\end{align*}
On the other hand, the Cauchy-Schwarz inequality implies that
\begin{displaymath}
  \abs{I_{22}} \leq \norm{\tilde{\omega}}_{V^{0}} \norm{\tilde{u}}_{V^{1}}.
\end{displaymath}
Hence $\norm{\tilde{h}}_{X}^{2}$ satisfies
\begin{displaymath}
  \frac{d}{dt} \norm{\tilde{h}}_{X}^{2} \leq C \Bigl\{ \norm{h_{1}}_{W^{1}} + \norm{h_{2}}_{W^{1}} + 1 \Bigr\} \norm{\tilde{h}}_{X}^{2}.
\end{displaymath}
Now by Gronwall's lemma,
\begin{displaymath}
  \norm{\tilde{h}(t)}_{X}^{2} \leq C \norm{\tilde{h}(0)}_{X}^{2} \exp \biggl\{ \int_{0}^{t} \Bigl[ \norm{h_{1}(s)}_{W^{1}} + \norm{h_{2}(s)}_{W^{1}} + 1 \Bigr]\,ds \biggr\}.
\end{displaymath}
Since $\tilde{h}(0) = 0$ and
\begin{displaymath}
  h_{1},\, h_{2} \in C_{w}([0,T]; H) \subset L^{\infty}([0,T]; H) \subset L^{1}([0,T]; W^{1}(S)),
\end{displaymath}
we conclude that $\tilde{h} \equiv 0$ on $[0,T]$, which proves the uniqueness of the solution.

To establish the strong continuity of the solution, we follow a standard argument which starts by showing that the solution determined by Theorem \ref{thm_abs_evol} is strongly right continuous at $t = 0$. Indeed, Theorem \ref{thm_abs_evol} implies that
$\norm{u(t)}_{H}^{2} \leq \rho(t)$ and, in particular, (see \eqref{eqn_abs_maj})
\begin{displaymath}
  \limsup_{t \to 0^{+}} \norm{h(t)}_{H}^{2} \leq \limsup_{t \to 0^{+}} \rho(t) = \norm{h_{0}}_{H}^{2}.
\end{displaymath}
On the other hand, the weak continuity of $h(t)$ at $t = 0$ implies that
\begin{displaymath}
  \norm{h_{0}}_{H}^{2} \leq \liminf_{t \to 0^{+}} \norm{h(t)}_{H}^{2}.
\end{displaymath}
Hence $\norm{h(t)}_{H} \to \norm{h(0)}_{H}$ as $t \to 0^{+}$, which establishes the strong right continuity of $h(t)$ at $t = 0$. To prove the right continuity of $h(t)$ at any $t_{0} \in [0,T]$, let $\tilde{h}(t)$ be the solution of the 1D model
\eqref{eqn_eat_1d} for $t \geq t_{0}$ with initial data $\tilde{h}(t_{0}) = h(t_{0})$. Then $\tilde{h}(t)$ is strongly right continuous at $t = t_{0}$. But the two solutions $h(t)$ and $\tilde{h}(t)$ must coincide for $t \geq t_{0}$ by uniqueness, so
$h(t)$ is strongly right continuous at $t = t_{0}$. This shows that $h(t)$ is strongly right continuous on $[0,T]$. Since the 1D model \eqref{eqn_eat_1d} is time-reversible, which is apparent from the two-sided estimate of the nonlinear pairing $\langle
h,A(h) \rangle$, it follows that $h(t)$ is strongly continuous on $[0,T]$. Finally, the higher regularity and strong continuity of $h_{t}(t)$ follows directly from Proposition \ref{prop_cal_ineq_1}, which asserts that
\begin{displaymath}
  h_{t}(t) = -A(h(t)) \in W^{m-1}(S),\qquad \forall t \in [0,T],
\end{displaymath}
and
\begin{multline*}
  \norm{h_{t}(t_{1}) - h_{t}(t_{2})}_{W^{m-1}} = \norm{A(h(t_{1})) - A(h(t_{2}))}_{W^{m-1}} \\
  \leq C \Bigl\{ \norm{h(t_{1})}_{W^{m}} + \norm{h(t_{2})}_{W^{m}} + 1 \Bigr\} \norm{h(t_{1}) - h(t_{2})}_{W^{m}},\qquad \forall t_{1},\, t_{2} \in [0,T].
\end{multline*}
The proof of Theorem \ref{thm_1d_loc_ex} is complete.

\subsection{Proof of Theorem \ref{thm_1d_reg}}\label{ssec_wp_reg}
We next prove Theorem \ref{thm_1d_reg}. The key of the proof is to find an estimate of the existence time $T$ that depends only on the low-norm $\norm{u_{0}}_{V^{2}},\ \norm{\omega_{0}}_{V^{1}}$ of the initial data. In \citet{kl1984}, this is
accomplished using the technique of norm compression. Here we give a different proof where the idea is to directly bound the high-norms of the solution in terms of its low-norms. In particular, we have

\begin{prop}
Let $m \geq 1$ be any positive integer and let
\begin{displaymath}
  h = (u,\omega) \in C([0,T]; W^{m}(S)) \cap C^{1}([0,T]; W^{m-1}(S))
\end{displaymath}
be a solution of \eqref{eqn_eat_1d} in class $CV^{m}$ on $[0,T]$, with the initial data $h_{0} \in W^{m}(S)$. Then
\begin{subequations}
\begin{equation}
  \max_{t \in [0,T]} \norm{h(t)}_{W^{m}} \leq M_{m}(T) \norm{h_{0}}_{W^{m}},
  \label{eqn_reg_hm_m}
\end{equation}
where $M_{m}(T)$ is a constant depending on $m,\ \norm{h_{0}}_{W^{\min(m,2)}}$, and
\begin{displaymath}
  M_{0}(T) = \exp \biggl\{ \int_{0}^{T} \norm{H\omega(t)}_{L^{\infty}}\,dt \biggr\}.
\end{displaymath}
In addition, there exists an absolute constant $C_{0}$ depending only on $L$ such that
\begin{equation}
  \norm{h(t)}_{W^{1}} \leq \frac{\re^{C_{0} t} \norm{h_{0}}_{W^{1}}}{1 - (\re^{C_{0} t}-1) \norm{h_{0}}_{W^{1}}},\qquad \forall t \in [0,T^{*}),
  \label{eqn_reg_hm_1}
\end{equation}
\end{subequations}
where $T^{*}$ is the first time at which the right-hand side of \eqref{eqn_reg_hm_1} becomes unbounded. \label{prop_reg_hm}
\end{prop}

\begin{proof}
We first derive an upper bound for $\norm{u_{z}}_{L^{\infty}}$ in terms of $\norm{H\omega}_{L^{\infty}}$. To begin with, we differentiate \eqref{eqn_eat_1d_u} with respect to $z$ (note that $u_{t} \in C([0,T]; V^{1}(S))$):
\begin{displaymath}
  u_{zt} + v u_{zz} = -v_{z} u_{z},
\end{displaymath}
and introduce the characteristic variable
\begin{displaymath}
  \frac{d}{dt} z(t) = v[z(t),t],\qquad z(0) = \xi \in [0,L].
\end{displaymath}
This leads to the ODE
\begin{displaymath}
  \frac{d}{dt} u_{z}[z(t),t] = -H\omega[z(t),t] \cdot u_{z}[z(t),t].
\end{displaymath}
Its solution is easily found to be
\begin{displaymath}
  u_{z}[z(t),t] = u_{0z}(\xi) \exp \biggl\{ -\int_{0}^{t} H\omega[z(s),s]\,ds \biggr\},
\end{displaymath}
hence
\begin{displaymath}
  \norm{u_{z}(t)}_{L^{\infty}} \leq \norm{u_{0z}}_{L^{\infty}} \exp \biggl\{ \int_{0}^{t} \norm{H\omega(s)}_{L^{\infty}}\,ds \biggr\} =: M_{0}(t) \norm{u_{0z}}_{L^{\infty}},
\end{displaymath}
where
\begin{displaymath}
  M_{0}(t) = \exp \biggl\{ \int_{0}^{t} \norm{H\omega(s)}_{L^{\infty}}\,ds \biggr\},
\end{displaymath}
which is the desired estimate.

Next, we derive an upper bound for the high-norm $\norm{h}_{W^{m}}$ in terms of $\norm{H\omega}_{L^{\infty}}$. To simplify the argument, we first assume $h$ is sufficiently smooth, e.g. $h$ belongs to class $CV^{m+1}$; then for each $1 \leq k \leq m$,
taking the $W^{k}$-inner product of equations \eqref{eqn_eat_1d_u}--\eqref{eqn_eat_1d_w} with $h$ yields
\begin{subequations}\label{eqn_reg_est}
\begin{equation}
  \frac{1}{2} \frac{d}{dt} \norm{h}_{W^{k}}^{2} = -(h,A(h))_{W^{k}} = -(u,v u_{z})_{V^{k+1}} - (\omega,v \omega_{z} - u_{z})_{V^{k}}.
  \label{eqn_reg_est_hk}
\end{equation}
Applying estimates \eqref{eqn_cal_ineq_2_tu2}--\eqref{eqn_cal_ineq_2_tw2} from Lemma \ref{lmm_cal_ineq_2} with $k \gets k+1,\ \tilde{u} = u$ and $\tilde{\omega} = \omega$ (again assuming a sufficiently smooth $h$), we obtain
\begin{align}
  \abs{(u,v u_{z})_{V^{k+1}}} & \leq C \norm{u}_{V^{k+1}} \Bigl\{ \norm{H\omega}_{L^{\infty}} \norm{u}_{V^{k+1}} + \norm{u_{z}}_{L^{\infty}} \norm{\omega}_{V^{k}} \Bigr\}, \label{eqn_reg_est_u} \\
  \abs{(\omega,v \omega_{z})_{V^{k}}} & \leq C \norm{\omega}_{V^{k}} \Bigl\{ \norm{H\omega}_{L^{\infty}} \norm{\omega}_{V^{k}} + (1-\delta_{k,1}) \norm{\omega_{z}}_{L^{\infty}} \norm{\omega}_{V^{k-1}} \Bigr\}. \label{eqn_reg_est_w}
\end{align}
\end{subequations}
For $k = 1$, \eqref{eqn_reg_est} implies
\begin{displaymath}
  \frac{d}{dt} \norm{h}_{W^{1}}^{2} \leq C \Bigl\{ \norm{H\omega}_{L^{\infty}} + \norm{u_{z}}_{L^{\infty}} + 1 \Bigr\} \norm{h}_{W^{1}}^{2} \leq K_{0} \Bigl\{ \norm{H\omega}_{L^{\infty}} + M_{0} \Bigr\} \norm{h}_{W^{1}}^{2},
\end{displaymath}
where $K_{0}$ is a constant depending on $\norm{h_{0}}_{W^{1}}$ (without loss of generality we assume $K_{0} \geq 2$). Then by Gronwall's lemma,
\begin{displaymath}
  \norm{h(t)}_{W^{1}} \leq M_{1}(t) \norm{h_{0}}_{W^{1}},\qquad \forall t \in [0,T],
\end{displaymath}
where
\begin{displaymath}
  M_{1}(t) = \exp \biggl\{ \frac{1}{2}\, K_{0} \int_{0}^{t} \Bigl[ \norm{H\omega(s)}_{L^{\infty}} + M_{0}(s) \Bigr]\,ds \biggr\} \geq M_{0}(t).
\end{displaymath}
For $k = 2$, \eqref{eqn_reg_est} implies
\begin{displaymath}
  \frac{d}{dt} \norm{h}_{W^{2}}^{2} \leq C \Bigl\{ \norm{H\omega}_{L^{\infty}} + \norm{u_{z}}_{L^{\infty}} + \norm{\omega}_{V^{1}} + 1 \Bigr\} \norm{h}_{W^{2}}^{2} \leq K_{1} \Bigl\{ \norm{H\omega}_{L^{\infty}} + M_{1} \Bigr\} \norm{h}_{W^{2}}^{2},
\end{displaymath}
where $K_{1}$ is a constant depending on $\norm{h_{0}}_{W^{1}}$ (without loss of generality we assume $K_{1} \geq K_{0}$). Then by Gronwall's lemma,
\begin{displaymath}
  \norm{h(t)}_{W^{2}} \leq M_{2}(t) \norm{h_{0}}_{W^{2}},\qquad \forall t \in [0,T],
\end{displaymath}
where
\begin{displaymath}
  M_{2}(t) = \exp \biggl\{ \frac{1}{2}\, K_{1} \int_{0}^{t} \Bigl[ \norm{H\omega(s)}_{L^{\infty}} + M_{1}(s) \Bigr]\,ds \biggr\} \geq M_{1}(t).
\end{displaymath}
Finally, for $2 < k \leq m$, \eqref{eqn_reg_est} becomes
\begin{displaymath}
  \frac{d}{dt} \norm{h}_{W^{k}}^{2} \leq C \Bigl\{ \norm{H\omega}_{L^{\infty}} + \norm{u_{z}}_{L^{\infty}} + \norm{\omega}_{V^{2}} + 1 \Bigr\} \norm{h}_{W^{k}}^{2} \leq K_{k-1} \Bigl\{ \norm{H\omega}_{L^{\infty}} + M_{2} \Bigr\} \norm{h}_{W^{k}}^{2},
\end{displaymath}
where $K_{k-1}$ is a constant depending on $k$ and $\norm{h_{0}}_{W^{2}}$. Gronwall's lemma then implies
\begin{displaymath}
  \norm{h(t)}_{W^{k}} \leq M_{k}(t) \norm{h_{0}}_{W^{k}},\qquad \forall t \in [0,T],
\end{displaymath}
where
\begin{displaymath}
  M_{k}(t) = \exp \biggl\{ \frac{1}{2}\, K_{k-1} \int_{0}^{t} \Bigl[ \norm{H\omega(s)}_{L^{\infty}} + M_{2}(s) \Bigr]\,ds \biggr\}.
\end{displaymath}
This establishes the high-norm estimate \eqref{eqn_reg_hm_m} for $\norm{h}_{W^{m}}$ under the assumption of sufficiently smooth $h$. To prove the low-norm estimate \eqref{eqn_reg_hm_1}, it suffices to note that $\norm{h}_{W^{1}}$ can alternatively be
estimated by
\begin{displaymath}
  \frac{d}{dt} \norm{h}_{W^{1}}^{2} \leq C \Bigl\{ \norm{H\omega}_{L^{\infty}} + \norm{u_{z}}_{L^{\infty}} + 1 \Bigr\} \norm{h}_{W^{1}}^{2} \leq 2C_{0} \Bigl\{ \norm{h}_{W^{1}} + 1 \Bigr\} \norm{h}_{W^{1}}^{2},
\end{displaymath}
where $C_{0}$ is an absolute constant depending only on $L$. The desired estimate then follows from Gronwall's lemma.

To complete the proof of the proposition, we need to rigorously justify the above formal manipulations, in particular \eqref{eqn_reg_est} for $k = m$. We achieve this by regarding $h = (u,\omega)$ as solutions of the \emph{linear} hyperbolic equation
\begin{subequations}\label{eqn_1d_lin}
\begin{equation}
  h_{t} + v h_{z} = f,\qquad f = (0,u_{z}),
  \label{eqn_1d_lin_h}
\end{equation}
with initial data
\begin{gather}
  h_{0} = (u_{0},\omega_{0}) \in W^{m}(S), \label{eqn_1d_lin_ic} \\
  \intertext{and velocity}
  v \in C([0,T]; V^{m+1}(S)) \cap C^{1}([0,T]; V^{m}(S)). \label{eqn_1d_lin_v}
\end{gather}
\end{subequations}
It is a standard result from the theory of linear hyperbolic equations that problem \eqref{eqn_1d_lin} has a unique solution in class $C([0,T]; W^{m}(S)) \cap C^{1}([0,T]; W^{m-1}(S))$, and by approximating the data $h_{0}$ using smooth functions,
\eqref{eqn_reg_est} and the resulting high-norm estimates can be readily established as shown above. The same estimate for $h$ then follows from a density argument.
\end{proof}

Now we are ready to prove Theorem \ref{thm_1d_reg}. Suppose
\begin{displaymath}
  h = (u,\omega) \in C([0,T]; W^{1}(S)) \cap C^{1}([0,T]; W^{0}(S))
\end{displaymath}
is a solution of the 1D model \eqref{eqn_eat_1d} in class $CV^{1}$ on $[0,T]$, with the initial data $h_{0} = (u_{0},\omega_{0}) \in W^{m}(S)$. By Theorem \ref{thm_1d_loc_ex}, there exists another $T_{m} > 0$ that may depend on $\norm{h_{0}}_{W^{m}}^{2}
= \norm{u_{0}}_{V^{m+1}}^{2} + \norm{\omega_{0}}_{V^{m}}^{2}$ such that \eqref{eqn_eat_1d} has a solution in class $CV^{m}$ on $[0,T_{m}]$:
\begin{displaymath}
  \tilde{h} = (\tilde{u},\tilde{\omega}) \in C([0,T_{m}]; W^{m}(S)) \cap C^{1}([0,T_{m}]; W^{m-1}(S)).
\end{displaymath}
By uniqueness $\tilde{h} = h$ on $[0,T_{m}]$, so $h$ belongs to class $CV^{m}$ on $[0,T_{m}]$. Now let $T_{m}^{*} \geq T_{m} > 0$ denote the first time at which $h$ ceases to be a solution in $CV^{m}$. We shall show that $T_{m}^{*} \geq T$. Suppose this
is not the case, i.e. $T_{m}^{*} < T$. Then by Proposition \ref{prop_reg_hm},
\begin{displaymath}
  \norm{h(t)}_{W^{m}} \leq M_{m}(T) \norm{h_{0}}_{W^{m}},\qquad \forall t \in [0,T_{m}^{*}),
\end{displaymath}
where $M_{m}(T)$ is a constant depending only on
\begin{displaymath}
  M_{0}(T) = \exp \biggl\{ \int_{0}^{T} \norm{H\omega(t)}_{L^{\infty}}\,dt \biggr\} \leq \exp \Bigl\{ CT \max_{t \in [0,T]} \norm{h(t)}_{W^{1}} \Bigr\}
\end{displaymath}
and $\norm{h_{0}}_{W^{\min(m,2)}}$. Consequently, $\norm{h(t)}_{W^{m}}$ is uniformly bounded on $[0,T_{m}^{*})$ with a bound depending only on $\norm{h_{0}}_{W^{m}}$ and $\max_{t \in [0,T]} \norm{h(t)}_{W^{1}}$, and by Theorem \ref{thm_1d_loc_ex}, there
exists, for each $t_{0} < T_{m}^{*}$, a $\delta > 0$ independent of $t_{0}$ such that \eqref{eqn_eat_1d} has a solution $\tilde{h}$ in $CV^{m}$ on $[t_{0},t_{0}+\delta]$ with initial data $\tilde{h}(t_{0}) = h(t_{0})$. By uniqueness $\tilde{h}$ and $h$
must coincide on $[t_{0},t_{0}+\delta]$, which then contradicts the assumption that $h$ cannot be continued in class $CV^{m}$ beyond $T_{m}^{*}$. This completes the proof of Theorem \ref{thm_1d_reg}.

\subsection{Proof of Theorem \ref{thm_1d_c_dep}}\label{ssec_wp_c_dep}
We next prove Theorem \ref{thm_1d_c_dep}. The key of the proof is to find, for any pair of solutions $h_{1} = (u_{1},\omega_{1}),\ h_{2} = (u_{2},\omega_{2})$ to \eqref{eqn_eat_1d}, an appropriate bound for the difference $\norm{A(h_{1}) -
A(h_{2})}_{W^{m}}$ where $A$ is the nonlinear operator defined in \eqref{eqn_1d_a}. Since this bound generally involves higher norms $\norm{h_{1}}_{W^{m+1}},\ \norm{h_{2}}_{W^{m+1}}$ of the solution, it is necessary to introduce some form of
regularizations so that the desired estimates carry through.

To this end, we follow the idea of \citet{kl1984} and introduce the family of smoothing operators
\begin{equation}
  J_{\epsilon} f = \eta_{\epsilon} * f,\qquad f \in V^{0}(S),\ \epsilon \in (0,1],
  \label{eqn_c_dep_m}
\end{equation}
where $\eta_{\epsilon}$ is a (smooth) approximation to identity on $V^{0}(S)$ (e.g. the Poisson kernel or, for $\epsilon = 1/n,\ n \in \mathbb{N}$, the Fej\'{e}r kernel), and
\begin{displaymath}
  f*g(z) = \frac{1}{L} \int_{0}^{L} f(z-y) g(y)\,dy = \frac{1}{L} \int_{0}^{L} g(z-y) f(y)\,dy
\end{displaymath}
denotes the convolution of two $L$-periodic functions $f$ and $g$. The following properties of $J_{\epsilon}$ are well known \citep[see, for example,][]{mb2002}:

\begin{prop}
Let $J_{\epsilon}$ be the smoothing operator (mollifier) defined in \eqref{eqn_c_dep_m}. Then for any $m,\, k \geq 0$ and $f \in V^{m}(S)$:
\begin{enumerate}
  \item $J_{\epsilon} f \in V^{m+k}(S)$;
  \item there exists constant $C$ depending only on $m,\ k$ and $L$ such that $\norm{J_{\epsilon} f}_{V^{m+k}} \leq C \epsilon^{-k} \norm{f}_{V^{m}}$;
  \item $\norm{J_{\epsilon} f - f}_{V^{m}} \to 0$ and, if $m \geq 1$, $\epsilon^{-1} \norm{J_{\epsilon} f - f}_{V^{m-1}} \to 0$ as $\epsilon \to 0^{+}$.
\end{enumerate}
\label{prop_c_dep_m}
\end{prop}

Returning to problem \eqref{eqn_eat_1d}, for any given data $h_{0} = (u_{0},\omega_{0}) \in W^{m}(S)$ with $m \geq 3$, we consider the smoothed problem with $h_{0}$ replaced by $h_{0}^{\epsilon} = J_{\epsilon} h_{0}$ and denote the corresponding
solutions by $h^{\epsilon} = (u^{\epsilon},\omega^{\epsilon})$. Similarly, for any given data $h_{0,j} = (u_{0,j},\omega_{0,j}) \in W^{m}(S)$, we consider the smoothed problem with $h_{0,j}$ replaced by $h_{0,j}^{\epsilon} = J_{\epsilon} h_{0,j}$ and
denote the corresponding solutions by $h_{j}^{\epsilon} = (u_{j}^{\epsilon},\omega_{j}^{\epsilon})$. Since, by Proposition \ref{prop_c_dep_m}, the set of data $\{h_{0}^{\epsilon}, h_{0,j}^{\epsilon}\}$ is uniformly bounded in $W^{m}(S)$ for all $\epsilon
\in (0,1]$ and $j \in \mathbb{N}$, it follows from Theorem \ref{thm_1d_loc_ex} that there exists a $T' > 0$ independent of $\epsilon$ and $j$ such that the solutions $\{h^{\epsilon}, h_{j}^{\epsilon}\}$ exist in class $CV^{m}$ on $[0,T']$. In view of
Proposition \ref{prop_reg_hm}, we may also choose $T'$ sufficiently small so that $\{h^{\epsilon}, h_{j}^{\epsilon}\}$ are uniformly bounded in $W^{m}(S)$ on $[0,T']$. Since each $h_{0}^{\epsilon}$ and $h_{0,j}^{\epsilon}$ belongs to $W^{m+1}(S)$,
Theorem \ref{thm_1d_reg} shows that $h^{\epsilon}$ and $h_{j}^{\epsilon}$ indeed belongs to class $CV^{m+1}$.

The key step in the proof of Theorem \ref{thm_1d_c_dep} is the following proposition, which establishes the uniform convergence of the smoothed solutions $h^{\epsilon}$ to $h$, and similarly the uniform convergence of $h_{j}^{\epsilon}$ to $h_{j}$, in
$W^{m}(S)$ as $\epsilon \to 0^{+}$.
\begin{prop}
Let $m \geq 3$ be a positive integer and let
\begin{displaymath}
  h = (u,\omega) \in C([0,T]; W^{m}(S)) \cap C^{1}([0,T]; W^{m-1}(S))
\end{displaymath}
be a solution of \eqref{eqn_eat_1d} in class $CV^{m}$ on $[0,T]$, with the initial data $h_{0} \in W^{m}(S)$. Let $h^{\epsilon}$ be the solution of the smoothed problem with initial data $h_{0}^{\epsilon} = J_{\epsilon} h_{0}$ where $J_{\epsilon}$ is the
smoothing operator defined in \eqref{eqn_c_dep_m}, and let $T' > 0$ be the common existence time of the solutions $\{h, h^{\epsilon}\}$ chosen as above. Then
\begin{displaymath}
  \max_{t \in [0,T']} \norm{h^{\epsilon}(t) - h(t)}_{W^{m}} \to 0\qquad \text{as}\qquad \epsilon \to 0^{+},
\end{displaymath}
and the convergence is uniform when $h_{0}$ varies over compact subsets of $W^{m}(S)$. \label{prop_c_dep_conv}
\end{prop}

\begin{proof}
For any $0 < \delta < \epsilon \leq 1$ and $2 \leq k \leq m$, we subtract the two equations satisfied by $h^{\epsilon},\ h^{\delta}$, take the $W^{k}$-inner product of the resulting equation with $\tilde{h} = (\tilde{u},\tilde{\omega}) =
(u^{\epsilon}-u^{\delta},\omega^{\epsilon}-\omega^{\delta})$, and use the observation that $\tilde{h} \in C^{1}([0,T']; W^{k})$ to obtain
\begin{displaymath}
  \frac{1}{2} \frac{d}{dt} \norm{\tilde{h}}_{W^{k}}^{2} = -(\tilde{h}, A(h^{\epsilon})-A(h^{\delta}))_{W^{k}} = I_{11} + I_{12} + I_{21} + I_{22} + I_{23},
\end{displaymath}
where
\begin{align*}
  I_{11} + I_{12} & = -(\tilde{u}, v^{\epsilon} u_{z}^{\epsilon} - v^{\delta} u_{z}^{\delta})_{V^{k+1}} = -(\tilde{u}, \tilde{v} u_{z}^{\epsilon})_{V^{k+1}} - (\tilde{u}, v^{\delta} \tilde{u}_{z})_{V^{k+1}},& \tilde{v} & = v^{\epsilon}-v^{\delta}, \\
  I_{21} + I_{22} & = -(\tilde{\omega}, v^{\epsilon} \omega_{z}^{\epsilon} - v^{\delta} \omega_{z}^{\delta})_{V^{k}} = -(\tilde{\omega}, \tilde{v} \omega_{z}^{\epsilon})_{V^{k}} - (\tilde{\omega}, v^{\delta} \tilde{\omega}_{z})_{V^{k}},& I_{23} & = (\tilde{\omega}, \tilde{u}_{z})_{V^{k}}.
\end{align*}
Since $h^{\epsilon},\, h^{\delta} \in W^{m+1}(S) \subset W^{k+1}(S)$, Lemma \ref{lmm_cal_ineq_2} applies with $k \gets k+1$, yielding
\begin{align*}
  \abs{I_{11}} & \leq C \norm{\tilde{u}}_{V^{k+1}} \Bigl\{ \norm{\tilde{\omega}}_{V^{0}} \norm{u^{\epsilon}}_{V^{k+2}} + \norm{u^{\epsilon}}_{V^{2}} \norm{\tilde{\omega}}_{V^{k}} \Bigr\}, \\
  \abs{I_{12}} & \leq C \norm{\tilde{u}}_{V^{k+1}} \Bigl\{ \norm{\omega^{\delta}}_{V^{1}} \norm{\tilde{u}}_{V^{k+1}} + \norm{\tilde{u}}_{V^{2}} \norm{\omega^{\delta}}_{V^{k}} \Bigr\}, \\
  \abs{I_{21}} & \leq C \norm{\tilde{\omega}}_{V^{k}} \Bigl\{ \norm{\tilde{\omega}}_{V^{0}} \norm{\omega^{\epsilon}}_{V^{k+1}} + \norm{\omega^{\epsilon}}_{V^{2}} \norm{\tilde{\omega}}_{V^{k-1}} \Bigr\}, \\
  \abs{I_{22}} & \leq C \norm{\tilde{\omega}}_{V^{k}} \Bigl\{ \norm{\omega^{\delta}}_{V^{1}} \norm{\tilde{\omega}}_{V^{k}} + \norm{\tilde{\omega}}_{V^{2}} \norm{\omega^{\delta}}_{V^{k-1}} \Bigr\}.
\end{align*}
Summing up these estimates and invoking the Cauchy-Schwarz inequality
\begin{displaymath}
  \abs{I_{23}} \leq \norm{\tilde{\omega}}_{V^{k}} \norm{\tilde{u}}_{V^{k+1}},
\end{displaymath}
we deduce
\begin{equation}
  \frac{d}{dt} \norm{\tilde{h}}_{W^{k}}^{2} \leq C \Bigl\{ \norm{h^{\epsilon}}_{W^{2}} + \norm{h^{\delta}}_{W^{1}} + 1 \Bigr\} \norm{\tilde{h}}_{W^{k}}^{2}
  + C \Bigl\{ \norm{h^{\epsilon}}_{W^{k+1}} + \norm{h^{\delta}}_{W^{k}} \Bigr\} \norm{\tilde{h}}_{W^{2}} \norm{\tilde{h}}_{W^{k}}.
  \label{eqn_c_dep_est}
\end{equation}

We shall now derive an estimate for $\norm{\tilde{h}}_{W^{m}}$ and use the result to show that $\{h^{\epsilon}\}$ is Cauchy in
\begin{displaymath}
  X := C([0,T']; W^{m}(S)) \cap C^{1}([0,T']; W^{m-1}(S)).
\end{displaymath}
To begin with, we set $k = 2$ in \eqref{eqn_c_dep_est} to obtain
\begin{displaymath}
  \frac{d}{dt} \norm{\tilde{h}}_{W^{2}}^{2} \leq C \Bigl\{ \norm{h^{\epsilon}}_{W^{3}} + \norm{h^{\delta}}_{W^{2}} + 1 \Bigr\} \norm{\tilde{h}}_{W^{2}}^{2}.
\end{displaymath}
Since $\{h_{0}^{\epsilon}\}$ is uniformly bounded in $W^{m}(S)$ for all $\epsilon \in (0,1]$:
\begin{displaymath}
  \sup_{\epsilon \in (0,1]} \norm{h_{0}^{\epsilon}}_{W^{m}} \leq K,
\end{displaymath}
and $T'$ is chosen sufficiently small, there exists, by Proposition \ref{prop_reg_hm}, a constant $K_{1}$ depending on $m,\ K$ and $T'$ such that
\begin{displaymath}
  \sup_{\epsilon \in (0,1]} \max_{t \in [0,T']} \norm{h^{\epsilon}(t)}_{W^{m}} \leq C \sup_{\epsilon \in (0,1]} \norm{h_{0}^{\epsilon}}_{W^{m}} \leq K_{1}.
\end{displaymath}
Using Gronwall's lemma and noting that $m \geq 3$, we then deduce
\begin{displaymath}
  \max_{t \in [0,T']} \norm{\tilde{h}(t)}_{W^{2}} \leq \re^{CK_{1} T'} \norm{\tilde{h}(0)}_{W^{2}}.
\end{displaymath}
Now setting $k = m$ in \eqref{eqn_c_dep_est} we obtain
\begin{displaymath}
  \frac{d}{dt} \norm{\tilde{h}}_{W^{m}} \leq C \Bigl\{ \norm{h^{\epsilon}}_{W^{2}} + \norm{h^{\delta}}_{W^{1}} + 1 \Bigr\} \norm{\tilde{h}}_{W^{m}}
  + C \Bigl\{ \norm{h^{\epsilon}}_{W^{m+1}} + \norm{h^{\delta}}_{W^{m}} \Bigr\} \norm{\tilde{h}}_{W^{2}}.
\end{displaymath}
By Proposition \ref{prop_reg_hm} and Proposition \ref{prop_c_dep_m}, there holds
\begin{displaymath}
  \sup_{\epsilon \in (0,1]} \max_{t \in [0,T']} \epsilon \norm{h^{\epsilon}(t)}_{W^{m+1}} \leq C \sup_{\epsilon \in (0,1]} \epsilon \norm{h_{0}^{\epsilon}}_{W^{m+1}} \leq K_{2},
\end{displaymath}
where $K_{2}$ is another constant depending on $m,\ K$ and $T'$, and (recall that $\delta < \epsilon$)
\begin{align*}
  \max_{t \in [0,T']} \epsilon^{-1} \norm{\tilde{h}(t)}_{W^{2}} & \leq \re^{CK_{1} T'} \epsilon^{-1} \norm{\tilde{h}(0)}_{W^{2}} \\
  & \leq \re^{CK_{1} T'} \Bigl\{ \epsilon^{-1} \norm{h_{0}^{\epsilon} - h_{0}}_{W^{2}} + \delta^{-1} \norm{h_{0}^{\delta} - h_{0}}_{W^{2}} \Bigr\} \to 0\qquad \text{as}\qquad \epsilon \to 0^{+}.
\end{align*}
Hence Gronwall's lemma implies that
\begin{displaymath}
  \max_{t \in [0,T']} \norm{\tilde{h}(t)}_{W^{m}} \leq \re^{CK_{1} T'} \biggl\{ \norm{\tilde{h}(0)}_{W^{m}} + CK_{2} \epsilon^{-1} \int_{0}^{T'} \norm{\tilde{h}(s)}_{W^{2}}\,ds \biggr\} \to 0\qquad \text{as}\qquad \epsilon \to 0^{+},
\end{displaymath}
which shows that $\{h^{\epsilon}\}$ is uniformly Cauchy in $W^{m}(S)$. To see $\{h_{t}^{\epsilon}\}$ is also uniformly Cauchy in $W^{m-1}(S)$, it suffices to recall from Proposition \ref{prop_cal_ineq_1} that $A(h)$ is strongly continuous from $W^{m}(S)$
to $W^{m-1}(S)$, i.e.
\begin{align*}
  \norm{\tilde{h}_{t}}_{W^{m-1}} & = \norm{A(h^{\epsilon}) - A(h^{\delta})}_{W^{m-1}} \\
  & \leq C \Bigl\{ \norm{h^{\epsilon}}_{W^{m}} + \norm{h^{\delta}}_{W^{m}} + 1 \Bigr\} \norm{\tilde{h}}_{W^{m}} \leq CK_{1} \norm{\tilde{h}}_{W^{m}}.
\end{align*}
Hence $\{h^{\epsilon}\}$ is Cauchy in
\begin{displaymath}
  X = C([0,T']; W^{m}(S)) \cap C^{1}([0,T']; W^{m-1}(S)),
\end{displaymath}
as claimed.

Since $X$ is complete (with the obvious choice of the norm), there exists a unique $\hat{h} \in X$ such that
\begin{displaymath}
  \max_{t \in [0,T']} \Bigl\{ \norm{h^{\epsilon}(t) - \hat{h}(t)}_{W^{m}} + \norm{h_{t}^{\epsilon}(t) - \hat{h}_{t}(t)}_{W^{m-1}} \Bigr\} \to 0\qquad \text{as}\qquad \epsilon \to 0^{+}.
\end{displaymath}
Since
\begin{align*}
  \norm{\hat{h}_{t} + A(\hat{h})}_{W^{m-1}} & = \limsup_{\epsilon \to 0^{+}} \norm{\hat{h}_{t} + A(\hat{h}) - h_{t}^{\epsilon} - A(h^{\epsilon})}_{W^{m-1}} \\
  & \leq \limsup_{\epsilon \to 0^{+}} \Bigl\{ \norm{\hat{h}_{t} - h_{t}^{\epsilon}}_{W^{m-1}} + CK_{1} \norm{\hat{h} - h^{\epsilon}}_{W^{m}} \Bigr\} = 0,
\end{align*}
$\hat{h}$ is also a solution of \eqref{eqn_eat_1d} in class $CV^{m}$. But by uniqueness, $\hat{h}$ must coincide with $h$ on $[0,T']$, so
\begin{displaymath}
  \max_{t \in [0,T']} \Bigl\{ \norm{h^{\epsilon}(t) - h(t)}_{W^{m}} + \norm{h_{t}^{\epsilon}(t) - h_{t}(t)}_{W^{m-1}} \Bigr\} \to 0\qquad \text{as}\qquad \epsilon \to 0^{+}.
\end{displaymath}
In addition, this convergence is uniform when $h_{0}$ varies over compact subsets of $W^{m}(S)$, since the convergence of $h_{0}^{\epsilon} = J_{\epsilon} h_{0}$ to $h_{0}$ is uniform over compact subsets of $W^{m}(S)$. Hence the proposition follows.
\end{proof}

Now we are ready to prove Theorem \ref{thm_1d_c_dep}. Since $\{h_{0}^{\epsilon}, h_{0,j}^{\epsilon}\}$ is compact in $W^{m}(S)$, for any given $\delta > 0$ there exists, by Proposition \ref{prop_c_dep_conv}, an $\epsilon \in (0,1]$ such that
\begin{displaymath}
  \max_{t \in [0,T']} \norm{h^{\epsilon}(t) - h(t)}_{W^{m}} < \frac{\delta}{3},\qquad \sup_{j \in \mathbb{N}} \max_{t \in [0,T']} \norm{h_{j}^{\epsilon}(t) - h_{j}(t)}_{W^{m}} < \frac{\delta}{3}.
\end{displaymath}
For this fixed $\epsilon$, a computation similar to the one leading to \eqref{eqn_c_dep_est} shows that
\begin{displaymath}
  \frac{d}{dt} \norm{\tilde{h}}_{W^{m}}^{2} \leq C \Bigl\{ \norm{h_{j}^{\epsilon}}_{W^{m+1}} + \norm{h^{\epsilon}}_{W^{m}} + 1 \Bigr\} \norm{\tilde{h}}_{W^{m}}^{2} \leq CK_{2} \epsilon^{-1} \norm{\tilde{h}}_{W^{m}}^{2},
\end{displaymath}
where $\tilde{h} = h_{j}^{\epsilon} - h^{\epsilon} \in W^{m+1}(S)$. Gronwall's lemma then implies that
\begin{displaymath}
  \max_{t \in [0,T']} \norm{\tilde{h}(t)}_{W^{m}} \leq \re^{CK_{2} \epsilon^{-1} T'} \norm{\tilde{h}(0)}_{W^{m}} < \frac{\delta}{3},
\end{displaymath}
provided that $j > j_{0}$ is sufficiently large. Hence
\begin{align*}
  \max_{t \in [0,T']} \norm{h_{j}(t) - h(t)}_{W^{m}} & \leq \max_{t \in [0,T']} \Bigl\{ \norm{h_{j}(t) - h_{j}^{\epsilon}(t)}_{W^{m}} \\
  &\qquad{}+ \norm{h_{j}^{\epsilon}(t) - h^{\epsilon}(t)}_{W^{m}} + \norm{h^{\epsilon}(t) - h(t)}_{W^{m}} \Bigr\} < \delta,\qquad \forall j > j_{0},
\end{align*}
which shows that
\begin{displaymath}
  \max_{t \in [0,T']} \norm{h_{j}(t) - h(t)}_{W^{m}} \to 0\qquad \text{as}\qquad j \to \infty.
\end{displaymath}
This completes the proof of Theorem \ref{thm_1d_c_dep}.

\subsection{Proof of Theorem \ref{thm_1d_bkm}}\label{ssec_wp_bkm}
Finally, we give the proof of Theorem \ref{thm_1d_bkm}. Suppose first that \eqref{eqn_1d_bkm} holds, i.e.
\begin{displaymath}
  \int_{0}^{T} \norm{H\omega(t)}_{L^{\infty}}\,dt = \infty,
\end{displaymath}
then necessarily
\begin{displaymath}
  \limsup_{t \to T^{-}} \norm{H\omega(t)}_{L^{\infty}} = \infty.
\end{displaymath}
But by Sobolev's imbedding theorem and Poincar\'{e}'s inequality,
\begin{displaymath}
  \norm{H\omega(t)}_{L^{\infty}} \leq C \norm{\omega(t)}_{V^{1}} \leq C \norm{\omega(t)}_{V^{m}},
\end{displaymath}
so
\begin{displaymath}
  \limsup_{t \to T^{-}} \norm{\omega(t)}_{V^{m}} = \infty.
\end{displaymath}
This shows that the solution cannot be continued in class $CV^{m}$ up to $t = T$.

Next, suppose that \eqref{eqn_1d_bkm} does not hold, i.e.
\begin{equation}
  \int_{0}^{T} \norm{H\omega(t)}_{L^{\infty}}\,dt < \infty.
  \label{eqn_1d_bkm_n}
\end{equation}
Then by Proposition \ref{prop_reg_hm},
\begin{displaymath}
  \norm{h(t)}_{W^{m}} \leq M_{m}(T) \norm{h_{0}}_{W^{m}},\qquad \forall t \in [0,T),\ h = (u,\omega),
\end{displaymath}
where $M_{m}(T)$ is a constant depending only on
\begin{displaymath}
  M_{0}(T) = \exp \biggl\{ \int_{0}^{T} \norm{H\omega(t)}_{L^{\infty}}\,dt \biggr\} < \infty
\end{displaymath}
and $\norm{h_{0}}_{W^{\min(m,2)}}$. Consequently, $\norm{h(t)}_{W^{m}}$ is uniformly bounded on $[0,T)$ with a bound depending only on $\norm{h_{0}}_{W^{m}}$ and $M_{0}(T)$, and by Theorem \ref{thm_1d_loc_ex}, there exists, for each $t_{0} < T$, a
$\delta
> 0$ independent of $t_{0}$ such that \eqref{eqn_eat_1d} has a solution $\tilde{h}$ in $CV^{m}$ on $[t_{0},t_{0}+\delta]$ with initial data $\tilde{h}(t_{0}) = h(t_{0})$. By uniqueness $\tilde{h}$ and $h$ must coincide on $[t_{0},t_{0}+\delta]$, which
then shows that $h$ can be continued in class $CV^{m}$ to $t = T+\frac{1}{2} \delta$. This completes the proof of Theorem \ref{thm_1d_bkm}.

\section*{Acknowledgement}
This work was supported in part by NSF FRG Grant DMS-1159138.

\end{document}